\documentclass[11pt,reqno]{amsart}
\usepackage[utf8]{inputenc}
\usepackage{amsmath}
\usepackage{amsfonts}
\usepackage{amssymb}
\usepackage{amsthm}
\usepackage{mathtools}
\usepackage{dsfont}
\usepackage{xcolor}
\usepackage{tikz}
\usepackage{enumitem}
\usepackage{relsize}
\usepackage{graphicx}
\usepackage{wrapfig}
\usetikzlibrary{patterns}
\usepackage{float}
\usepackage[left=3cm,right=3cm,top=3cm,bottom=4cm]{geometry}
\usepackage{ulem}

\title[Discrete-time Hawkes process with inhibition]{(Almost) complete characterization of stability of a discrete-time Hawkes process with inhibition and memory of length two}
\date{}

\author[M. Costa]{Manon Costa}
\address{Manon Costa\\ Institut de Math\'ematiques de Toulouse, CNRS UMR 5219 \\
	Universit\'e Toulouse III Paul Sabatier
	\\ 118 route
	de Narbonne, F-31062 Toulouse cedex 09.} \email{manon.costa@math.univ-toulouse.fr}

\author[P. Maillard]{Pascal Maillard}
\address{Pascal Maillard\\Institut Universitaire de France and Institut de Math\'ematiques de Toulouse, CNRS UMR 5219 \\
	Universit\'e Toulouse III Paul Sabatier
	\\ 118 route
	de Narbonne, F-31062 Toulouse cedex 09.} \email{pascal.maillard@math.univ-toulouse.fr}

\author[A. Muraro]{Anthony Muraro}
\address{Anthony Muraro\\ Institut de Math\'ematiques de Toulouse, CNRS UMR 5219 \\
	Universit\'e Toulouse III Paul Sabatier
	\\ 118 route
	de Narbonne, F-31062 Toulouse cedex 09.} \email{anthony.muraro@math.univ-toulouse.fr}

\theoremstyle{definition}

\theoremstyle{plain}

\newtheorem{lemma}{Lemma}

\newtheorem{prop}{Proposition}
\newtheorem{theorem}{Theorem}

\newcommand{\R}{\mathbb{R}}
\newcommand{\C}{\mathbb{C}}

\newcommand{\N}{\mathbb{N}}

\renewcommand{\P}{\mathbb{P}}
\newcommand{\E}{\mathbb{E}}

\newcommand*{\couple}[2]{( #1, #2 )}

\newcommand*{\proba}[1]{\mathbb{P}\left( #1 \right)}

\tikzset{
    xmin/.store in=\xmin, xmin/.default=-3, xmin=-3,
    xmax/.store in=\xmax, xmax/.default=3, xmax=3,
    ymin/.store in=\ymin, ymin/.default=-3, ymin=-3,
    ymax/.store in=\ymax, ymax/.default=3, ymax=3,
}

\begin{document}
\maketitle

\begin{center}
	\textsc{Universit\'e de Toulouse}
	\smallskip
\end{center}
\begin{abstract}
We consider a Poisson autoregressive process whose parameters depend on the past of the trajectory. We allow these parameters to take negative values, modelling inhibition. More precisely, the model is the stochastic process $(X_n)_{n\ge0}$ with parameters $a_1,\ldots,a_p \in \R$, $p\in\N$ and $\lambda \ge 0$, such that for all $n\ge p$, conditioned on $X_0,\ldots,X_{n-1}$, $X_n$ is Poisson distributed with parameter 
\[
\left(a_1 X_{n-1} + \cdots + a_p X_{n-p} + \lambda \right)_+
\]
This process can be regarded as a discrete time Hawkes process with inhibition and a memory of length $p$. In this paper we initiate the study of necessary and sufficient conditions of stability for these processes, which seems to be a hard problem in general.
We consider specifically the case $p = 2$, for which we are able to classify the asymptotic behavior of the process for the whole range of parameters, except for boundary cases. In particular, we show that the process remains stochastically bounded whenever the solution to the linear recurrence equation $x_n = a_1x_{n-1} + a_2x_{n-2} + \lambda$ remains bounded, but the converse is not true. Furthermore, the criterion for stochastic boundedness is not symmetric in $a_1$ and $a_2$, in contrast to the case of non-negative parameters, illustrating the complex effects of inhibition. 
\end{abstract}
\bigskip

\medskip

\medskip

\section{Introduction}
The motivation of this paper is to pave the way for obtaining sufficient and necessary conditions for the stability of non-linear Hawkes processes with inhibition.
Hawkes processes are a class of point processes used to model events that have mutual influence over time. They were initially introduced by Hawkes (1971) (\cite{hawkes_spectra_1971}, \cite{hawkes_cluster_1974}) and are now used in a variety of fields such as finance, biology, and neuroscience.

More precisely, a Hawkes process $(N_t^h)_{t\in \R} = \left(N^h([0,t]\right))_{t\in \R}$ is defined by its initial condition on $(-\infty,0]$ and its stochastic conditional intensity denoted by $\Lambda$, characterized by : 
\begin{equation}
\label{Lambda}
\Lambda(t) = \phi \left( \lambda + \int_{-\infty}^t h(t-s)N^h(\text{d}s)\right),
\end{equation}
where $\lambda >0$, $h : \R_+ \to \R$ and $\phi : \R \to \R_+$ are measurable, deterministic functions (see \cite{daley_introduction_2006} for further details).
 The function $h$ is called the \textit{reproduction function}, and contains the information of the behaviour of the process throughout time. 
In the case where $\phi$ is non-decreasing, the sign of the function $h$ encodes for the type of time dependence: when $h$ is non-negative, the process is said to be \textit{self-exciting}; when $h$ is signed : the negative values of $h$ can then be seen as self-inhibition (see \cite{cattiaux_limit_2021} and \cite{costa_renewal_2020}). The case where $h\ge0$ and $\phi=id$ is called the linear case.
Considering signed functions $h$ requires to add non-linearity by the mean of a function $\phi$ which ensures that the intensity remains positive.
In this paper, we will focus on the particular case where $\phi = (\cdot)_+$ is the ReLU function defined on $\R$ by $(x)_+ = \max(0,x)$.
\medskip

Several authors have established sufficient conditions on $h$ ensuring the existence of a stable version of this process. For signed $h$, Brémaud and Massoulié \cite{bremaud_stability_1996} proved that a stable version of the process exists if $\|h\|_1<1$. In \cite{costa_renewal_2020}, Costa \textit{et al.} proved that it is sufficient to have $\|h^+\|_1<1$, where $h^+(x)=\max(h(x),0)$ using a coupling argument.  Unfortunately, this sufficient criterion does not take into account the effect of inhibition, captured by the negative part of $h$.
Going further is difficult because non-linearity breaks the direct link between the function $h$ and the probabilistic structure of the Hawkes process. Recent results have been obtained by \cite{raad_stability_2020} for a two dimensional non-linear Hawkes process with weighted exponential kernel, modeling the case of two populations of interacting neurons including both inhibition and excitation. The author provide a criterion on the weight function matrix for stability exploiting the Markovian structure of the Hawkes process in that case. It is noteworthy that the stability condition  \cite[Assumption 1.2]{raad_stability_2020} is similar to the case $\mathcal{R}_2$ of this paper, by reinterpreting the meaning of our parameters to correspond to those of the model described in \cite{raad_stability_2020}. Our work focuses on a simpler process due to its discrete-time nature, yet the significance of our study lies in providing an almost complete classification of its asymptotic behaviour without requiring assumptions on the parameter values of the model. \medskip

In order to get an intuition on the results that one might obtain on Hawkes processes, we choose to consider a simplified, discrete analog of those processes. Namely we will study an autoregressive process $(\Tilde X_n)_{n\ge 1}$ with initial condition $(\Tilde X_0, \dots, \Tilde X_{-p+1})$ where $p \in \{ 1, 2, \dots \}$, and such that:
\begin{equation}
\label{defDiscreteHawkesP}
\forall n \geq 1, \quad \Tilde X_n \sim \mathcal{P}\left( \phi\left( a_1\Tilde X_{n-1} + \cdots + a_p\Tilde X_{n-p} + \lambda \right) \right),
\end{equation}
where $\mathcal{P(\rho)}$ denotes the Poisson distribution with parameter $\rho$, and $a_1,\ldots,a_p$ are real numbers.

In the linear case ($a_1,\ldots,a_p$ non-negative, and $\phi(x)=x$) these integer-valued processes are called INGARCH processes, and have already been studied in (\cite{ferland_integer-valued_2006}, \cite{fokianos_interventions_2010}), where a necessary condition for existence and stability of this class of processes has been derived and can be written $\sum_{i=1}^p {a_i}<1$. Furthermore, the link between Hawkes processes and autoregressive Poisson processes has already been made for the linear case : M. Kirchner proved that the discretized autoregressive process (with $p = +\infty$) converges weakly to the associated Hawkes process (see \cite{kirchner2016hawkes} for details). Although this convergence has only been demonstrated in the linear case, i.e., with positive $a_i$, it seemed valuable to us to understand the modifications induced by the presence of inhibition on the asymptotic behavior of these processes. An analogous discrete process has been proposed by Seol \cite{seol2015limit} using an autoregressive structure based on Bernouilli random variables. 

In order to explore the effect of inhibition, we consider signed values for the parameters $a_1,\ldots,a_p$. In this article, we focus on the specific case of $p=2$, so that our model of interest can be written as : 
\begin{equation}
\label{defDiscreteHawkes2}
\forall n \geq 2, \quad \Tilde X_n \sim \mathcal{P}\left( \left( a\Tilde X_{n-1} + b\Tilde X_{n-2} + \lambda \right)_+\right),
\end{equation}
with $a,b \in \R$ and $\Tilde X_0, \Tilde X_1 \in \N$. (In this paper we will use the convention $0 \in \mathbb{N}$).

In this paper, the most important result is the classification of the process defined in \eqref{defDiscreteHawkes2}.
Note that the complete characterization of the behavior of this simple process is difficult due to the variety of behaviors observed. We will prove that the introduction of non-linearity through the ReLU function makes the process more stable relatively to its linear counterpart, in the sense that the parameter space 
$(a,b) \in \R^2$ for which the linear process $y_{n+1}=ay_{n}+by_{n-1}+\lambda$ admits a stationary version is strictly a subset of the parameter space for which the non-linear process admits a stationary version ( see Appendix A.). Our results also enlighten the complex role of inhibition and in particular the asymmetric role of $a$ and $b$ associated with the range at which inhibition occurs. Our work suggests the existence of complex algebraic and geometric structures that are likely to play an important role in the more general case of a memory of order $p$.
In order to obtain our results, we used a wide range of probabilistic tools, corresponding to the variety of the behaviours of the trajectories of the process, depending on the parameters of the model.

\section{Notations, definitions and results}

\subsection{Definition and main result}

Let $a,b \in \mathbb{R}$ and $\lambda > 0$. We consider a discrete time process  $(\tilde{X}_n)_{n \geq 1}$ with initial condition $(\tilde{X}_0, \tilde{X}_{-1})$ such that the following holds for all $n\ge 1$:
\begin{equation*}
\text{conditioned on $\tilde{X}_{-1},\ldots,\tilde{X}_{n-1}$: }\tilde{X}_n\sim \mathcal{P}\left( \left(a \tilde{X}_{n-1} + b \tilde{X}_{n-2} + \lambda \right)_+ \right), 
\end{equation*}
where $(\cdot)_+$ is the ReLU function defined on $\mathbb{R}$ by $(x)_+ := \max(0,x)$.

As we said previously, some papers have already been dealing with the linear version of this process : if $a$ and $b$ are non-negative, the parameter of the Poisson random variable in \eqref{defDiscreteHawkes2} is also non-negative, and the ReLU function vanishes. In this case, Proposition 1 in \cite{ferland_integer-valued_2006} states that the process is a second-order stationary process if $a+b<1$. This weak stationarity ensures that the mean, variance and autocovariance are constant with time. \medskip


Let us define the function 
\[
b_c(a) = \begin{cases}
1 & a \le 0\\
1-a & a\in(0,2)\\
-\frac{a^2}{4} & a\ge 2
\end{cases}
\]
and define the following sets (see Figure~\ref{fig:RT}):
\begin{align}
\label{eq:def_R}
\mathcal R &= \left\{(a,b)\in\R^2: b<b_c(a)\right\}\\
\label{eq:def_T}
\mathcal T &= \left\{(a,b)\in\R^2: b>b_c(a)\right\}.
\end{align}

Our main result is the following
\begin{theorem}\ 
\label{thm:main}
\begin{itemize}
    \item If $(a,b) \in \mathcal{R}$, then the sequence $(\Tilde X_n)_{n\ge0}$ converges in law as $n\to\infty$. 
    \item if $(a,b) \in \mathcal{T}$, then the sequence $(\Tilde X_n)_{n\ge0}$ satisfies that
almost surely 
$$\Tilde X_n+\Tilde X_{n+1}\underset{n\to\infty}{\longrightarrow}+\infty\,.$$
\end{itemize}
\end{theorem}

This result derives from the study of the natural Markov chain associated with $\Tilde X_n$ that is defined by:
\begin{equation}
\label{MarkovChain}
X_n := (\Tilde X_{n}, \Tilde X_{n-1}), \quad n\ge0\,.
\end{equation}
Before giving more details about the behaviour of $(X_n)_{n\ge0}$, let us comment on Theorem \ref{thm:main}. In particular, we stress that the condition for convergence in law is not symmetrical in $a$ and $b$. More precisely, for any $a\in\R$, the sequence $(\Tilde X_n)$ can be tight provided that $b$ is chosen small enough, but the converse is not true as soon as $b>1$. This induces that inhibition has a stronger regulating effect when occurs after an excitation, rather than before. 

The question of the critical behavior of the process on the boundary $\{b=b_c(a)\}$ remains open and presents us a difficult question for further work.

\subsection{The associated Markov chain}
\label{sec:associated_markov_chain}

As mentioned above, the main part of the article is devoted to the study of a Markov chain $(X_n)$ which encodes the time dependency of $(\Tilde X_n)$. We will rely on the recent treatment by Douc, Moulines, Priouret and Soulier \cite{dmps_markov_chains} for results about Markov chains. In particular, we use their notion of irreducibility, which is weaker than the usual notion of irreducibility typically found in textbooks on Markov chains (on discrete state space). Thus, a Markov chain is called \emph{irreducible} if there exists an \emph{accessible state}, i.e., a state that can be reached with positive probability from any other state. Following Douc, Moulines, Priouret and Soulier \cite{dmps_markov_chains}, we refer to the usual notion of irreducibility (i.e., every state is accessible) as \emph{strong irreducibility}.


The transition matrix of the Markov chain $(X_n)_{n\ge0}$ defined in \eqref{MarkovChain} is thus given for $(i,j,k,l)\in\N^4$ by:

\begin{equation}
\label{transitionP}
P\left( \couple{i}{j}, \couple{k}{\ell} \right) = \delta_{i \ell} \dfrac{e^{-s_{ij}}s_{ij}^k}{k !},
\end{equation}
where :
$$s_{ij} := (ai+bj+\lambda)_+ \text{ and } \delta_{ij} := \begin{cases}
1 & \text{if $i=j$} \\ 
0 & \text{else}
\end{cases}.$$

In other words, starting from a state $(i,j)$, the next step of the Markov chain will be $(k,i)$ where $k \in \N$ is the realization of a Poisson random variable with parameter $s_{ij}$. In particular, if $s_{ij} = 0$, then the next step of the Markov chain is $(0,i)$.

Since the probability that a Poisson random variable is zero is strictly positive, it is possible to reach the state $(0,0)$ with positive probability from any state in two steps. In particular, the state $(0,0)$ is accessible and the Markov chain is irreducible. Furthermore, the Markov chain is aperiodic \cite[Section~7.4]{dmps_markov_chains}, since $P((0,0),(0,0)) = e^{-\lambda} >  0$. Note that strong irreducibility may not hold (see Proposition~\ref{prop:irr} in the appendix).


Recall the definition of the sets $\mathcal{R}$ and $\mathcal{T}$ in \eqref{eq:def_R} and \eqref{eq:def_T}. 

\begin{theorem}\ 
\label{thm:classification}
\begin{itemize}
\item Let $(a,b)\in \mathcal R$. Then the Markov chain $(X_n)_{n\ge0}$ is \emph{geometrically ergodic}, i.e., it admits an invariant probability measure $\pi$ and there exists $\beta > 1$, such that for every initial state,
\[
\beta^n d_{TV}(\operatorname{Law}(X_n),\pi) \to 0,\quad n\to\infty,
\]
where $d_{TV}$ denotes total variation distance.
\item Let $(a,b)\in\mathcal T$. Then the Markov chain is transient, i.e., every state is visited a finite number of times almost surely, for every initial state. 
\end{itemize}
\end{theorem}

\begin{figure}[!ht]
\includegraphics[scale=0.75]{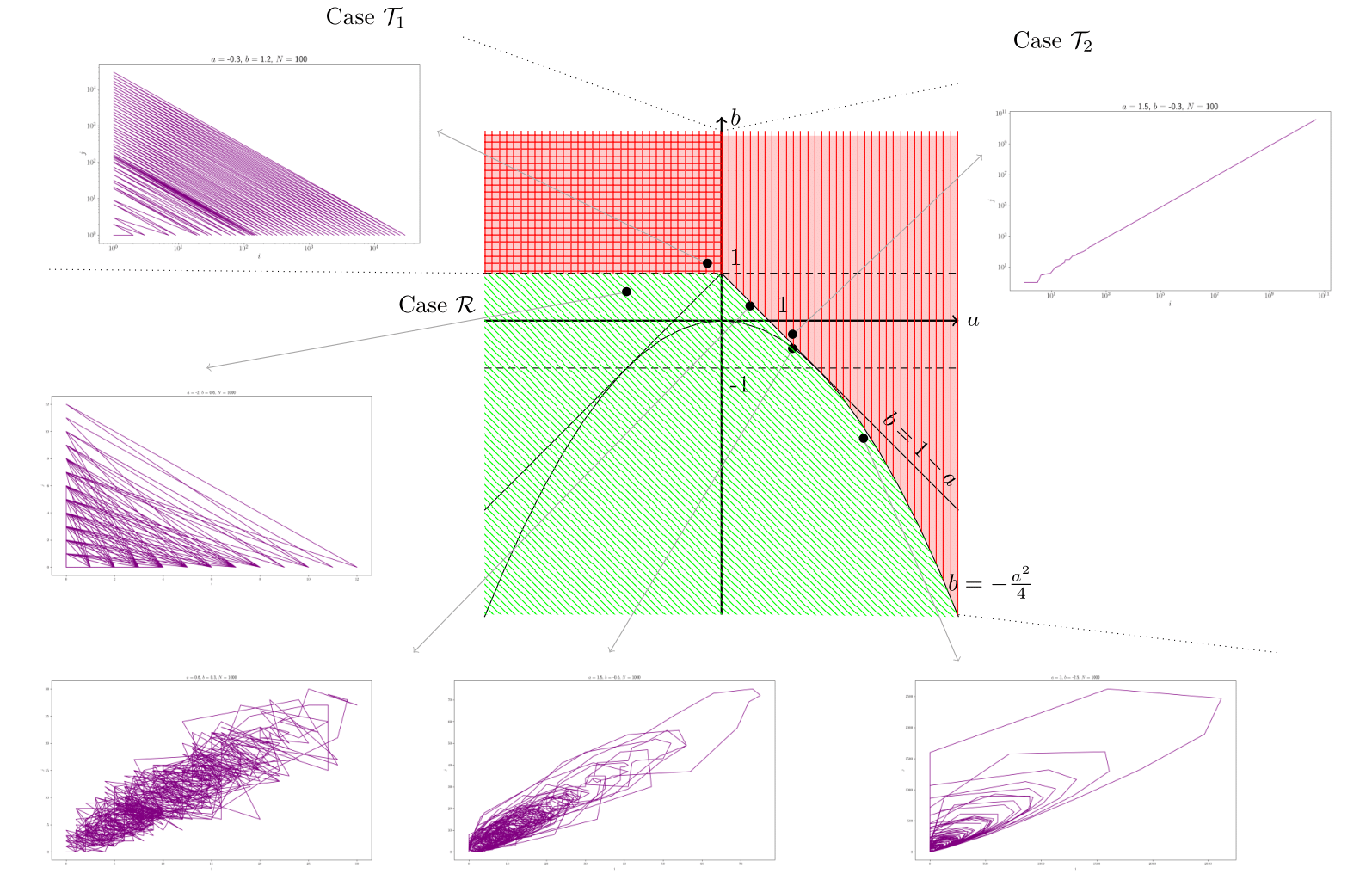}
\caption{\label{fig:RT}The partition of the parameter space described in Theorem \ref{thm:classification}. The green region corresponds to $\mathcal{R}$, while the red region corresponds to $\mathcal T$. The smaller figures are typical trajectories of the Markov chain $(X_n)_{n\ge0}$ for each region of the parameter space. In the all the simulations, we chose $\lambda=1$. The region delineated by a dashed triangular line corresponds to the region of parameter space for which the linear recurrence equation $y_{n} = ay_{n-1} + by_{n-2} + \lambda$ is bounded for all $n \in \N$, for any given $y_0, y_1 \in \R$. See Appendix A for more details.}
\end{figure}


Theorem \ref{thm:main} is a simple consequence of this result. Indeed, in the case of $(a,b)\in\mathcal{R}$, the convergence in law of $\tilde{X}_n$ simply derives from the convergence in law of $X_n$ since $\tilde{X}_n$ is the first coordinate of $X_n$. In the transient case,  $(a,b)\in\mathcal{T}$ the result in Theorem \ref{thm:main} simply derives from the fact that $||X_n||_1\to_{n\to\infty}\infty$ almost surely.
\medskip

The rest of the article is devoted to the proof of Theorem \ref{thm:classification}. We first focus on the recurrent case in Section \ref{sec:proof_thm2}, then on the transient case in Section \ref{sec:transience}. Throughout this paper, we will provide typical trajectories of the considered cases. For the sake of clarity, we have plotted the realizations of $(X_n)_{n=0,\dots,N}$ by connecting its successive realizations. Unless otherwise stated for coherence purpose, we always set $X_0 = (0,0)$ and $\lambda = 1$ for our plots.

\section{Proof of Theorem \ref{thm:classification}: recurrence}
\label{sec:proof_thm2}

In this section, we prove the recurrence part of Theorem \ref{thm:classification}. The proof goes by exhibiting three functions satisfying Foster-Lyapounov drift conditions for different ranges of the parameters $(a,b)$ covering the whole recurrent regime $\mathcal R$. 
\subsection{Foster-Lyapounov drift criteria}

 Drift criteria are powerful tools that have been introduced by Foster \cite{foster_stochastic_1953}, and deeply studied, and popularized, by Meyn and Tweedie \cite{meyn_markov_2009}, among others. These drift criteria allow to prove convergence to the invariant measure of Markov chains and yield explicit rates of convergence. We use here the treatment from Douc, Moulines, Priouret and Soulier \cite{dmps_markov_chains}, which is influenced by Meyn and Tweedie \cite{meyn_markov_2009}, but is more suitable for Markov chains which are irreducible but not strongly irreducible.
 
 A set of states $C\subset \N^2$ is called \emph{petite} \cite[Definition~9.4.1]{dmps_markov_chains}, if there exists a state $x_0\in \N^2$ and a probability distribution $(p_n)_{n\in\N}$ on $\N$ such that
 \[
 \inf_{x\in C}\sum_{n\in\N} p_n P^n(x,x_0) > 0,
 \]
 where we recall that $P^n(x,x_0)$ is the $n$-step transition probability from $x$ to $x_0$. Since the Markov chain $(X_n)_{n\ge0}$ is irreducible, any finite set is petite (take $x_0$ to be the accessible state) and any finite union of petite sets is petite \cite[Proposition~9.4.5]{dmps_markov_chains}.
 
Let $V : \N^2 \to [1,\infty)$ be a function, $\varepsilon\in(0,1]$, $K<\infty$ and $C\subset \N^2$ a set of states. We say that the \emph{drift condition} $D(V,\varepsilon,K,C)$ is satisfied if 
\begin{equation*}
\Delta V(x) := \mathbb{E}_x [V(X_1)-V(X_0)] \leq -\varepsilon V(x)+K\boldsymbol 1_C,
\end{equation*}
where $\mathbb{E}_x[~\cdot~] = \mathbb{E}[~\cdot ~|~ X_0 = x]$.
It is easy to see that this condition implies condition $D_g(V,\lambda,b,C)$ from \cite[Definition~14.1.5]{dmps_markov_chains}, with $\lambda = 1-\varepsilon$ and $b = K$.

\begin{prop}
\label{prop:Lyapounov}
Assume that the drift condition $D(V,\varepsilon,K,C)$ is verified for some $V$, $\varepsilon$, $K$ and $C$ as above and assume that $C$ is petite. Then there exists $\beta > 1$ and a probability measure $\pi$ on $\N^2$, such that, for every initial state $x\in \N^2$,
\[
\beta^n\times \sum_{y\in \N^2} V(y)|\P_x(X_n = y)-\pi(y)| \to 0,\quad n\to\infty.
\]
In particular, for every initial state $x\in \N^2$,
\[
\beta^n d_{TV}(\operatorname{Law}(X_n),\pi)\to 0,\quad n\to\infty
\]
and $\pi$ is an invariant probability measure for the Markov chain $(X_n)_{n\ge0}$.
\end{prop}
\begin{proof}
As mentioned in Section~\ref{sec:associated_markov_chain}, the Markov chain is irreducible and aperiodic. The first statement then follows by combining parts (ii) and (a) of Theorem~15.1.3 in \cite{dmps_markov_chains} with the remark preceding Corollary~14.1.6 in \cite{dmps_markov_chains}. The second statement immediately follows, noting that $V\ge 1$.
\end{proof}

\medskip

We will consider separately the following ranges of the parameters:
\begin{align*}
 \mathcal R_1 &= \{(a,b)\in \R^2: a,b < 1\text{ and }a+b<1\}\,,\\
 \mathcal R_2 &= \{a>0\text{ and }a^2 + 4b < 0\}\,,\\ 
 \mathcal R_3 &=  \{1\le a<2\text{ and } {-}1 < b < 1-a\}\,.
\end{align*}

\begin{figure}[!h]
\includegraphics[scale=0.9]{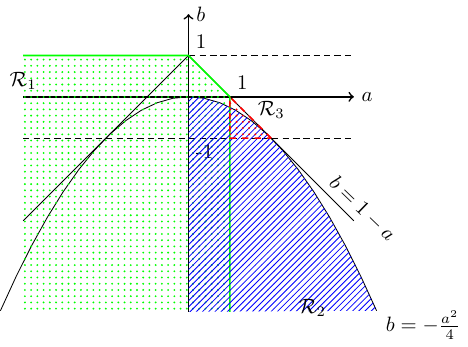}
\caption{Illustration of the three zones of parameters on which the proof of ergodicity will be carried.}
\label{fig:recurrent_regime}
\end{figure}
We then have $\mathcal R = \mathcal R_1\cup\mathcal R_2\cup \mathcal R_3,$ see Figure~\ref{fig:recurrent_regime}.

\subsection{Case $\mathcal R_1$}
\label{subsec:R1}


This case is the natural extension of the results that have been already proved for the linear process (see Proposition 1 in \cite{ferland_integer-valued_2006}).

Let $V : \mathbb{N}^2 \to \mathbb{R}_+$ be the function defined by 
$$V \couple{i}{j} := \alpha i + \beta j + 1,$$ 
where $\alpha,\beta >0$ are parameters to be chosen later. 

We then have that $V \couple{i}{j} \geq 1$ for all $\couple{i}{j} \in \mathbb{N}^2$. We look for $\varepsilon > 0$ such that $\Delta V(x) + \varepsilon V(x) \leq 0$ except for a finite number of $x \in \N^2$. 

Let $\varepsilon > 0$ to be properly chosen later. Then, 
\begin{align*}
\Delta V \couple{i}{j} + \varepsilon V \couple{i}{j} &= \sum_{k \in \mathbb{N}}\dfrac{e^{-s_{ij}}s_{ij}^k}{k !}(\alpha k + \beta i + 1) - (\alpha i + \beta j + 1) + \varepsilon(\alpha i + \beta j + 1)\\
&= \alpha s_{ij} + i(\beta - \alpha + \alpha \varepsilon) + j(\beta \varepsilon - \beta) + \varepsilon \,,
\label{drift geometric ergodic}
\end{align*}

Note that $s_{ij} = 0$ or $s_{ij} = ai + bj + \lambda > 0$. In both cases, $\Delta V + \varepsilon V$ is a linear function of $\couple{i}{j} \in \N^2$. We will thus choose $\alpha, \beta$  such that the coefficients of $\Delta V + \varepsilon V$ are negative, so there will be only a finite number of $\couple{i}{j}$ that satisfies $\Delta V \couple{i}{j} + \varepsilon V \couple{i}{j} \geq 0$.

Let us first consider couples $\couple{i}{j}$ such that $s_{ij}=0$. According to the above, it is sufficient to have :
$$\left\{ \begin{array}{ll} \beta - \alpha + \alpha \varepsilon < 0 \\ \beta \varepsilon - \beta < 0 \end{array} \right. \Longleftrightarrow \left\{ \begin{array}{ll} \beta < \alpha(1-\varepsilon) \\ \varepsilon < 1 \end{array} \right..$$

In the sequel, we impose $\varepsilon <1$.

If $s_{ij} = ai+bj+\lambda > 0$, then :
$$\Delta V\couple{i}{j}+ \varepsilon V \couple{i}{j} =  i(\alpha a - \alpha + \beta + \alpha \varepsilon) + j(\alpha b + \beta \varepsilon - \beta) + \lambda \alpha + \varepsilon.$$
For the same reasons as before, it is sufficient to have $\alpha,\beta > 0$ such that :
$$\left\{ \begin{array}{ll} \alpha a - \alpha + \beta + \alpha \varepsilon< 0 \\ \alpha b + \beta \varepsilon - \beta < 0 \end{array} \right. \Longleftrightarrow \left\{ \begin{array}{ll}\beta < \alpha ( 1 - a - \varepsilon)  \\ \beta > \dfrac{\alpha b}{1-\varepsilon} \hspace{1 cm} (\mbox{since } \varepsilon < 1) \end{array} \right..$$
Let $\alpha := 1$. With the above statements we thus want to choose $\beta, \varepsilon > 0$ such that :
$$\left\{ \begin{array}{ll} \dfrac{b}{1-\varepsilon} < \beta < 1-a-\varepsilon \\ \beta < 1-\varepsilon \end{array} \right. \hspace{0.5cm} \textit{i.e.,} \hspace{0.5cm} \dfrac{b}{1-\varepsilon} < \beta < \min \{ 1-a-\varepsilon, 1 - \varepsilon \}.$$
Recall that $a+b<1$, so it is possible to find $\varepsilon_0 \in (0,1)$ small enough so that :
$$ \forall \Tilde{\varepsilon} \leq \varepsilon_0, ~\dfrac{b}{1-\Tilde{\varepsilon}} < 1 -a-\Tilde{\varepsilon}.$$ 

\begin{itemize}[label=$-$]
\item If $a \geq 0$, then $\min \{ 1-a-\varepsilon, 1 - \varepsilon \} = 1 - a - \varepsilon$ and since $a < 1$, we can choose $\varepsilon \leq \varepsilon_0$ so small that $\dfrac{b}{1-\varepsilon} < \beta < \min \{ 1-a-\varepsilon, 1 - \varepsilon\}$ on one hand, and $1-a-\varepsilon > 0$ on the other hand. It is thus possible to choose $\beta > 0$ such that : 
$$\dfrac{b}{1-\varepsilon} < \beta < \min \{ 1-a-\varepsilon, 1 - \varepsilon \}.$$
\item If $a < 0$, then $\min \{ 1-a-\varepsilon, 1 - \varepsilon \} = 1 - \varepsilon$.
Since $b<1$, it is possible to set $\varepsilon \leq \varepsilon_0$ so small that $b < (1 - \varepsilon)^2$. Hence we have $\dfrac{b}{1-\varepsilon} < 1 - \varepsilon$, so that it is possible to choose $\beta > 0$ that satisfies our constrains. 
\end{itemize}

Note that $\Delta V\couple{0}{0} = \lambda > 0$. Hence, with $\alpha, \beta, \varepsilon > 0$ chosen as above, we have that :
$$\Delta V\couple{i}{j}\le  -\varepsilon V \couple{i}{j} \mbox{ except for a finite number of states } \couple{i}{j} \in \mathbb{N}^2.$$ 

This proves that a drift condition $D(V,\varepsilon,C)$ holds for finite set $C$, which yields the result. 

\subsection{Case $\mathcal R_2$}
\label{subsec:R2}

In this section, we assume that $a>0$ and $a^2+4b<0$. 
\begin{figure}[h!]
\centering
\includegraphics[scale=0.3]{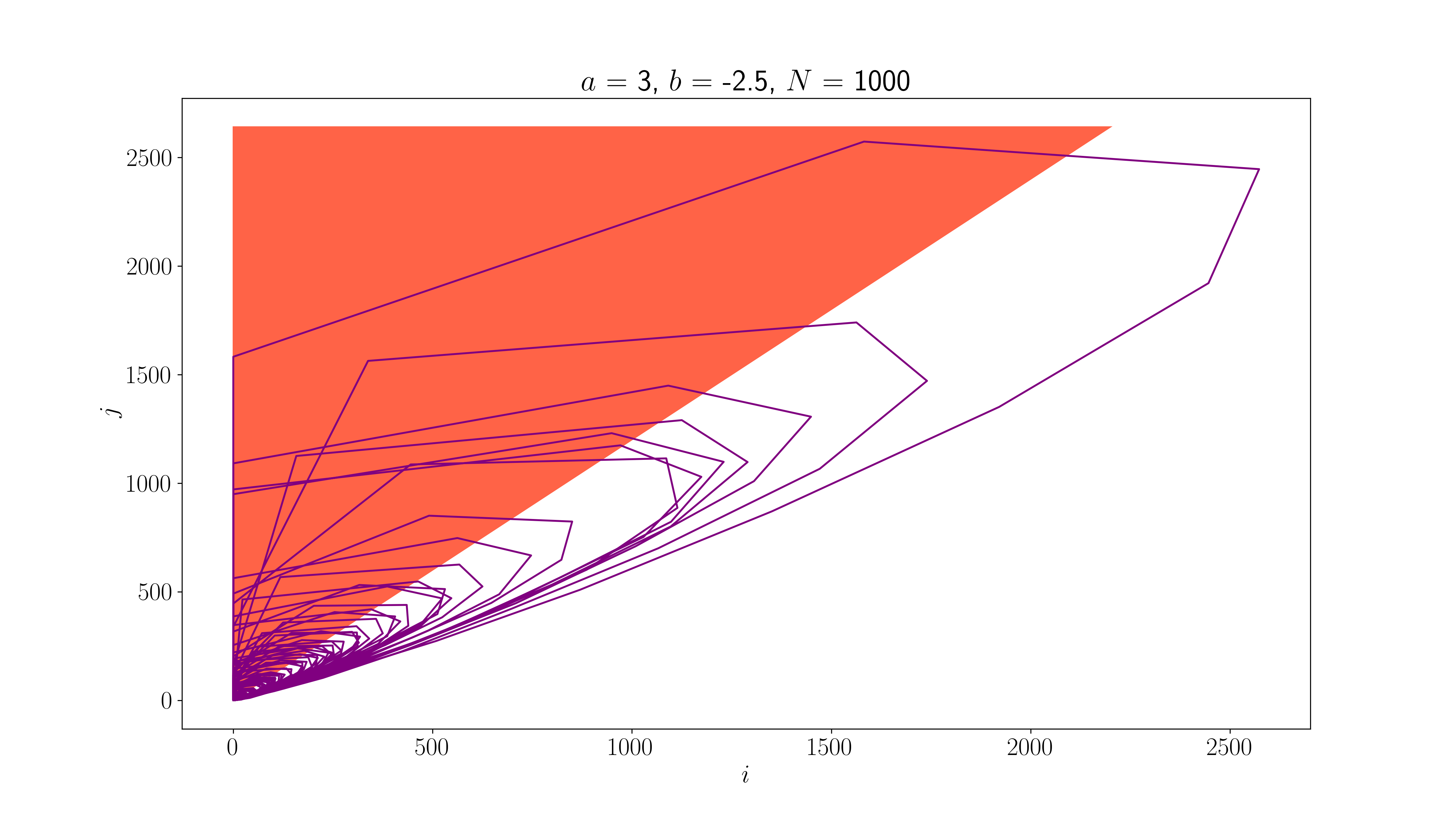}
\caption{An illustration of Case $\mathcal R_2$. Here, the parameters are $a = 3, b = -2.5$ and $N=1000$. In red, the set $A$ of couples $\couple{i}{j}$ such that $s_{ij}=s_{0i}=0$.}
\label{recurrenceasuptwo}
\end{figure}




The Lyapounov function we will consider is the following one:
$$\forall \couple{i}{j} \in \N^2, \quad V \couple{i}{j} = \dfrac{i}{j+1}+1.$$

Before getting into the details, a remark about this function. While we initially discovered it by trial and error, it has an interesting geometric interpretation. As seen in Figure~\ref{recurrenceasuptwo}, in case $\mathcal R_2$, the macroscopic trajectories of the Markov chain tend to turn counterclockwise until they hit the $j$-axis and eventually get pulled back to $\couple{0}{0}$. This provides a heuristic understanding of why $V$ should be a Lyapounov function. Indeed, it is an increasing function of the angle between the vector $(i,j)$ and the $j$-axis, and therefore $V(X_n)$ should have a tendency to decrease whenever $X_n$ is far away from the $j$-axis.

We now turn to the details. We will need to distinguish the region $A$ of the states $\couple{i}{j}$ where $s_{ij}=0$ (shown in red in Figure~\ref{recurrenceasuptwo}):
\begin{equation}
    \label{def:A}
A := \left\{ \couple{i}{j} \in \mathbb{N}^2 : s_{ij} = 0 \right\} = \left\{ \couple{i}{j} \in \mathbb{N}^2 : ai+bj+\lambda \le 0 \right\}\,.
\end{equation}
We have the following lemma:
\begin{lemma}
\label{lem:A_petite}The set $A$ is petite.
\end{lemma}
\begin{proof}
By definition of $A$, we have $s_{ij} = 0$ for all $(i,j)\in A$, hence $P((i,j),(0,i)) = 1$. Furthermore, for every $i\in \N$, since $b<-a^2/4 < 0$, we have
\[
P((0,i),(0,0)) = e^{-s_{0i}} = e^{-(\lambda+bi)_+} \ge e^{-\lambda}.
\]
It follows that 
$$
\inf_{(i,j)\in A} P^2((i,j),(0,0)) \ge e^{-\lambda} > 0,
$$
which shows that $A$ is petite.
%
 \end{proof}





\begin{lemma}
\label{lem:driftR2}
There exists a finite set $C\subset \N^2$ and $\varepsilon\in(0,1)$, such that the drift condition $D(V,\varepsilon,K,A\cup C)$ is satisfied for some $K<\infty$.
\end{lemma}
\begin{proof}
Since $\dfrac{a^2}{4}+b < 0$, there exists $\varepsilon \in (0,1)$ small enough such that $\frac{(a+\varepsilon)^2}{4}+b(1-\varepsilon)<0$.

Consider $(i,j) \not \in A$, and compute :
$$\Delta V \couple{i}{j} + \varepsilon V(i,j) = \dfrac{(\varepsilon-1)i^2+bj^2+(a+\varepsilon)ij + L_1(i,j)}{(i+1)(j+1)},$$
where $L_1(i,j)$ is a polynomial of degree 1.

On the numerator we recognize a quadratic form, and as $\frac{(a+\varepsilon)^2}{4}+b(1-\varepsilon)<0$, we have that this quadratic form is negative-definite. Thus, there is only a finite number of $(i,j) \not \in A$ such that $\Delta V(i,j) +\varepsilon V(i,j) > 0$. We define $C \subset \N^2 \setminus A$ to be the finite set of such $(i,j)$.

Note that for every $(i,j)\in A$, we have
\[
\Delta V \couple{i}{j} + \varepsilon V(i,j)\le \mathbb E_{(i,j)}[V(X_1)] = V(0,i) = 1.
\]
Hence, setting $K = 1\vee \max_{x\in C} \mathbb E_x[V(X_1)]\in [1,\infty)$, the finiteness of $K$ following from the fact that $C$ is finite, we have that the drift condition $D(V,\varepsilon,K,A\cup C)$ is satisfied.






\begin{figure}
\begin{center}
\includegraphics[scale=0.4]{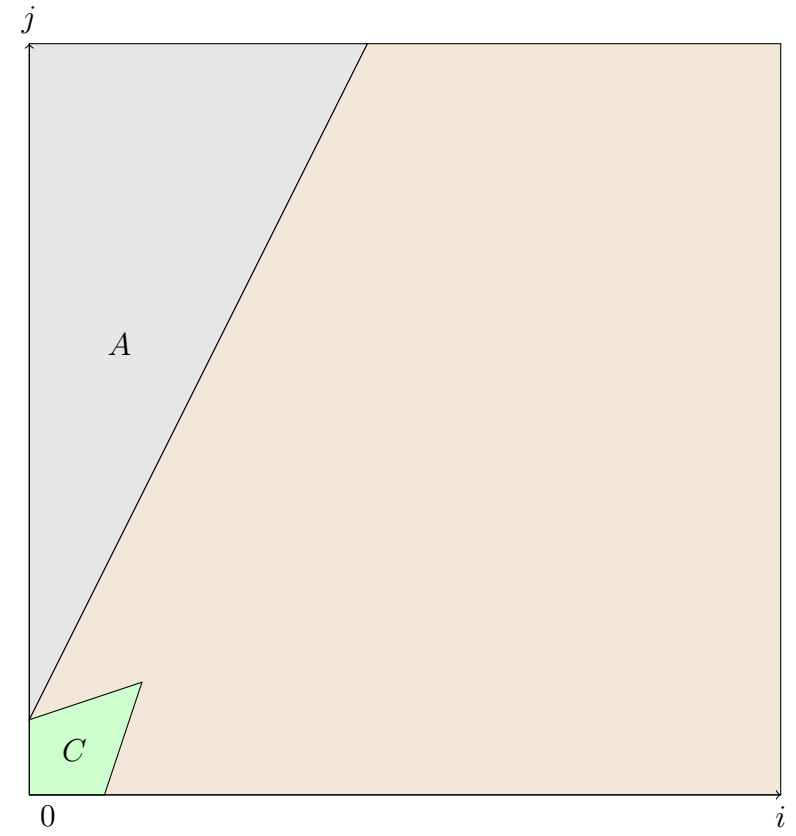}
\end{center}
\caption{\label{fig:R2_statespace}Graphical representation of the sets $A,C$ described above.}
\end{figure}

Figure~\ref{fig:R2_statespace} illustrates the cutting of the state space that we just described.
\end{proof}

In the case $\mathcal{R}_2$, by Lemma~\ref{lem:A_petite} and Lemma~\ref{lem:driftR2}, we can now apply Proposition~\ref{prop:Lyapounov}. 
Note that $A\cup C$ is petite because $A$ is petite (Lemma~\ref{lem:A_petite}), $C$ is finite, hence petite, and the union of two petite sets is again petite. This yields the proof of the case R2 of Theorem~\ref{thm:classification}.

\subsection{Case $\mathcal R_3$}
\label{subsec:R3}

To finish the proof of Theorem~\ref{thm:classification}, it suffices to consider parameters $a$ and $b$ such that $1\le a<2$ and $-\frac{a^2}{4} < b < 1-a$. However, for the sake of conciseness,  we will prove the ergodicity of the Markov chain on a larger space, namely $\mathcal R_3$. As a consequence, this case will cover some parameters sets which have already been considered in case $\mathcal R_2$. Note that this does not represent any issue in our strategy of proof. The choice of $\mathcal{R}_3$ will become clearer later on. 

We will thus assume here that $1\le a<2$ and $-1<b<1-a$. Let us denote by $V$ the following function :  
$$\forall (i,j) \in \mathbb{N}^2, \quad V \couple{i}{j} := 1 + \left(i^2 - aij + \dfrac{b^2+1}{2} j^2\right)\mathds{1}_{A^c}(i,j) .$$

First notice that the quadratic form in $V$ is positive-definite. 

Indeed, if $1\leq a < 2$, then $b^2 > (1-a)^2$ and :
 $$4 \times \dfrac{b^2+1}{2} - a^2 > 2(1-a)^2+2-a^2 = (a-2)^2 > 0.$$ 
Thus, the function $V$ satisfies $V \geq 1$.


Compute, for $(i,j) \not \in A$ and $\varepsilon \in (0,1)$ to be properly chosen later :
\begin{align*}
    \Delta V(i,j) + \varepsilon V(i,j) &= \sum_{k=0}^\infty \dfrac{e^{-s_{ij}} s_{ij}^k}{k!} V(k,i) + (\varepsilon-1)V(i,j) \\
    &\leq \sum_{k=0}^\infty \dfrac{e^{-s_{ij}} s_{ij}^k}{k!} \left( 1+\left(k^2 - aki + \dfrac{b^2+1}{2} i^2\right)\right) + (\varepsilon-1)V(i,j) \\
    &= s_{ij}(s_{ij}+1)-ais_{ij}+\dfrac{b^2+1}{2}i^2+(\varepsilon - 1)V(i,j) + 1\\
    &= \left( \dfrac{b^2-1}{2}+\varepsilon \right)i^2+a(b+1-\varepsilon)ij+\left( \dfrac{b^2(1+\varepsilon)+\varepsilon-1}{2}\right)j^2 + L_2(i,j),
\end{align*}
where $L_2(i,j)$ is a polynomial of degree 1. 

We want to choose $\varepsilon \in (0,1)$ such that the above quadratic form is negative-definite, that is, such that :
\begin{equation}
\label{condition_drift_R3}
\dfrac{b^2-1}{2}+\varepsilon < 0, \quad \text{ and } \quad  \left(\dfrac{b^2-1}{2}+\varepsilon\right) \left(\dfrac{b^2(1+\varepsilon)+\varepsilon-1}{2}\right)-\dfrac{a^2}{4}(b+1-\varepsilon)^2 > 0.
\end{equation}

On the one hand, we have $b^2-1<0$. On the other hand, the second inequality in \eqref{condition_drift_R3} can be written as follows :
$$(b^2-1)^2-a^2(b+1)^2 + k_{\varepsilon, a,b} > 0,$$
where $k_{\varepsilon,a,b} \in \R$ satisfies $k_{\varepsilon,a,b} \underset{\varepsilon \to 0}{\longrightarrow} 0$.

In addition, note that : 
$$ (a,b) \in \mathcal{R}_3 \Longrightarrow (b^2-1)^2 - a^2(b+1)^2 > 0.$$

From the foregoing, we deduce that there exists $\varepsilon \in (0,1)$ small enough such that both conditions of \eqref{condition_drift_R3} are satisfied. Thus, there is only a finite number of $(i,j) \not \in A$ such that $\Delta V(i,j) +\varepsilon V(i,j) > 0$. We define $C \subset \N^2 \setminus A$ to be the finite set of such $(i,j)$.

Finally, similarly as in Lemma~\ref{lem:A_petite}, the set $A$ is petite, because $b <1-a \leq 0$.  Furthermore similarly as in the case $\mathcal R_2$, for all $(i,j)\in A$, $\mathbb E_{(i,j)}(V(X_1))=V(0,i)$ is bounded, since $(0,i)\in A$ except for a finite number of $i$. Since the set $C$ is finite, we have that $A \cup C$ is a petite set and up to an adequate choice of $K$ the drift condition $D(V,\varepsilon,K, A\cup C)$ is satisfied.

\section{Proof of Theorem~\ref{thm:classification}: transience}
\label{sec:transience}

In this section, we show that the Markov chain $(X_n)_{n\ge0}$ is transient in the regime $\mathcal T$ of the parameters. We will distinguish between the following two cases:

\begin{itemize}
\item Case T1 : $a < 0, b>1$ (section \ref{subsec:T1}).
\item Case T2 : $0\leq a<2$ and $a+b>1$ or $a \geq 2$ and $a^2+4b > 0$ (section \ref{subsec:T2}).
\end{itemize}

In both cases, we will apply the following lemma:

\begin{lemma}
\label{lem:strategy_transience}
Let $S_1, S_2,\ldots$ be a sequence of subsets of $\N^2$ and $0<m_1 < m_2<\dots$ an increasing sequence of integers. Suppose that
\begin{enumerate}
    \item On the event $\bigcap_{n\ge 1} \{X_{m_n} \in S_n\}$, we have $X_n \ne (0,0)$ for all $n\ge 1$,
    \item $\P_{(0,0)}(X_{m_1}\in S_1) > 0$ and for all $n\ge 1$ and every $x\in S_n$, we have $\P_x(X_{m_{n+1}-m_n} \in S_{n+1}) > 0$.
    \item There exist $(p_n)_{n\ge 1}$ taking values in $[0,1]$ and such that $\sum_{n\ge 1} (1-p_n) < \infty$, such that
    \[
    \forall n\ge 1: \forall x\in S_n: \P_x(X_{m_{n+1}-m_n} \in S_{n+1}) \ge p_n.
    \]
\end{enumerate}
Then the Markov chain $(X_n)_{n\ge0}$ is transient.
\end{lemma}
\begin{proof}
Since $(0,0)$ is an accessible state, it is enough to show that
\[
\P_{(0,0)}(X_n \ne (0,0)\,,\forall n\ge 1) > 0.
\]
Using assumption 1, it is sufficient to prove that
\begin{equation}
\label{eq:strategy_transience_1}
\P_{(0,0)}(X_{m_n} \in S_n\,,\forall n\ge 1) > 0.
\end{equation}
By assumption 3, there exists $n_0\ge 1$ such that $\prod_{n\ge n_0} p_n > 0$. It follows that for every $x\in S_{n_0}$,
\[
\P_x(X_{m_n-m_{n_0}}\in S_{n}\,,\forall n>n_0) \ge \prod_{n\ge n_0} p_n > 0.
\]
Furthermore, by assumption 2, we have that
\[
\P_{(0,0)}(\forall n\le n_0\,, X_{m_n}\in S_n) > 0.
\]
Combining the last two inequalities yields \eqref{eq:strategy_transience_1} and finishes the proof.
\end{proof}

\subsection{Case T1}
\label{subsec:T1}

In this region of parameters, the Markov chain eventually reaches the $i$ and $j$ axes. Indeed, since $a<0$, if $(X_n)$ hits a state $\couple{i}{0}$ with $i \geq -\frac{\lambda}{a}$, as $s_{i0} = (ai+\lambda)_+ = 0$, the next step of the Markov chain will be $\couple{0}{i}$. Afterwards, the Markov chain will hit the state $\couple{\mathcal{P}(bi+\lambda)}{0}$, with $bi+\lambda > i$. Consequently, to follow the example, if we focus on the $i$ axe, starting from $\couple{k}{0}$ with $k$ big enough, the Markov chain will return in two steps to a state $\couple{k'}{0}$ belonging to the $i$-axe, satisfying $k' > k$ with high probability. 
\begin{figure}[h!]
\includegraphics[scale=0.3]{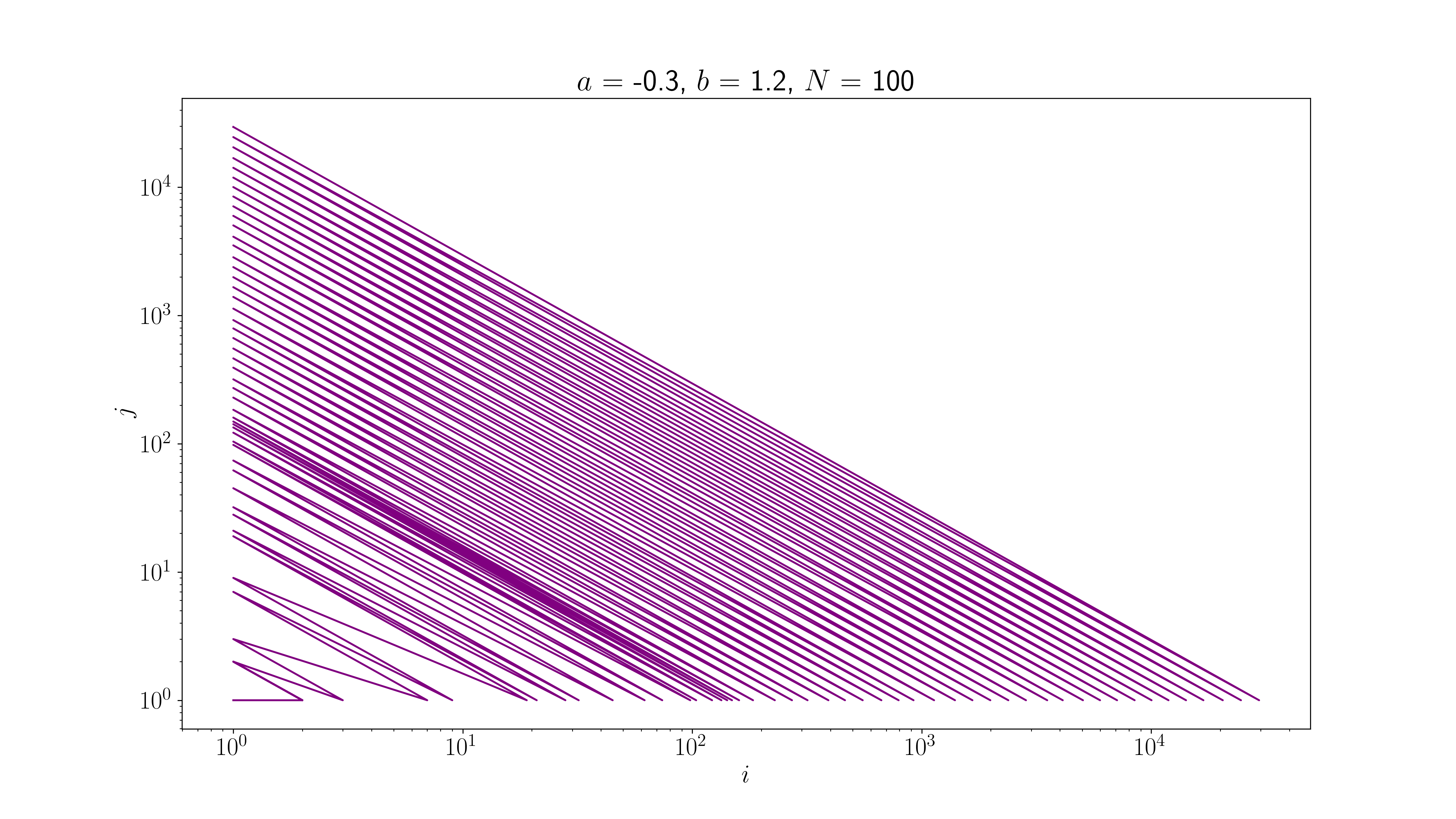}
\caption{Log-log plot of a typical trajectory of $(X_n)$, to make the erratic behaviour of the first points of the Markov chain more visible. Here, the parameters are $a = -0.3, b=1.2$ and $N=100$.}
\end{figure}

In order to formalize these observations, it is very natural to consider the Markov chain induced by the transition matrix $P^2$, namely $(X_{2n+1})_{n\geq 0}$. 
For $i \geq \dfrac{-\lambda}{a}, \quad s_{i0}=0$ and thus  : 
\begin{equation}
\label{poisson law}
\begin{aligned}
\mathbb{P}\left( X_{2n+1} = \couple{k}{0} ~\big| ~ X_{2n-1} = \couple{i}{0}, i \geq \dfrac{-\lambda}{a}\right) &= \dfrac{e^{-s_{0 i}} s_{0 i}^k}{k!} \\
&= \dfrac{e^{-(bi+\lambda)} (bi+\lambda)^k}{k!}.
\end{aligned}
\end{equation}
Note that if $a \leq -\lambda$, this results holds for $i \in \mathbb{N}$.

Equation \eqref{poisson law} means that if $\Tilde X_{2n-1} \geq -\frac{\lambda}{a}$, and $\Tilde X_{2n-2} = 0$, then $\Tilde X_{2n} = 0$, and $\Tilde X_{2n+1}$ is a Poisson random variable with parameter $b\Tilde X_{2n-1} + \lambda$.

Let us now prove our statement. 

\begin{proof}[Proof of the transience of $(X_n)$ when $a<0$ and $b>1$]
%
%
%
%
%
%
%

Fix $r\in (1,b)$. We wish to apply Lemma~\ref{lem:strategy_transience} with
\[
m_n = 2n-1,\quad n\ge 1
\]
and
\[
S_n = \{(i,0)\in \N: i \ge r^n\}.
\]
We verify that the assumptions (1)-(3) from Lemma~\ref{lem:strategy_transience} hold. For the first assumption, note that if $X_{2n-1} = (i,0) \in S_n$, then $X_{2n} = (j,i)$ for some $j$, hence $X_{2n-1} \ne (0,0)$ and $X_{2n} \ne (0,0)$ since $i\ge1$. In particular, assumption (1) holds.

We now verify that the second assumption holds.
For states $x,y\in \N^2$, write $x\to_1 y$ if $\P_x(X_1 = y) > 0$.  Furthermore, for $S\subset \N^2$, write $x\to_1 S$ if $x\to_1 y$ for some $y\in S$. Note that $(0,0) \to_1 (i,0)$ for every $i\in\N$, so that $(0,0)\to_1 S_1$. Now, for every $i\in \N$, we have $(i,0)\to_1 (0,i)$, and then, because $b>0$, $(0,i)\to_1 (j,0)$ for every $j\in\N$. In particular, from every $x\in S_n$, we can indeed reach $S_{n+1}$ in two steps. Hence, the second assumption is verified as well.

We now prove the third assumption. We claim that there exists $n_0 \in\N$, such that the following holds:
\begin{align}
\label{eq:toshow1}
    \forall n\geq n_0, \forall x\in S_n: \P_x(X_2\in S_{n+1}) \ge 1 - \dfrac{b}{(b-r)^2r^{n}}.
\end{align}
To prove \eqref{eq:toshow1}, first note that according to the earlier remark on \eqref{poisson law}, if $n_0$ is chosen such that $r^{n_0} \geq -\lambda/a$, then starting from a state $(i,0)$ with $i \geq r^{n_0}$, we have $\Tilde{X}_1 = 0$ almost surely and $\Tilde{X}_{2}\sim \mathcal{P}(bi + \lambda)$. Therefore, if $n\ge n_0$ and $i \ge r^n \ge r^{n_0}$,
\begin{align*}
1 - \P_{(i,0)}(\Tilde X_2 \ge r^{n+1}, \Tilde X_1 = 0) 
&= \P_{(i,0)}(\Tilde X_2 < r^{n+1})\\
&\le \proba{\mathcal{P}(bi + \lambda)<r^{n+1}}\\
&\leq \proba{\mathcal{P}(br^{n}) <r^{n+1}} \\
&= \proba{\mathcal{P}(br^{n})-br^{n} <r^{n}(r-b)} \\
&\leq \color{black} \proba{|\mathcal{P}(br^{n})-br^{n}| > r^{n}(b-r)} \\
&\leq \dfrac{b}{(b-r)^2r^{n}},
\end{align*}
by the Bienaymé-Chebychev inequality. This proves \eqref{eq:toshow1}. Now, \eqref{eq:toshow1} implies,
\[
\forall x\in S_n: \P_x(X_2\in S_{n+1}) \ge p_n \coloneqq \left(1 - \dfrac{b}{(b-r)^2r^{n}}\right)_+,
\]
and
\[
\sum_{n\ge 1} (1-p_n) \le \sum_{n\ge1} \dfrac{b}{(b-r)^2r^{n}} < \infty.
\]
This proves that the third assumption of Lemma~\ref{lem:strategy_transience} holds. The lemma then shows that the Markov chain is transient.
\end{proof}

\subsection{Case T2: $0\leq a<2$ and $a+b > 1$ or $a\geq 2$ and $a^2+4b > 0$}
\label{subsec:T2}

\begin{figure}
\centering
\includegraphics[scale=0.3]{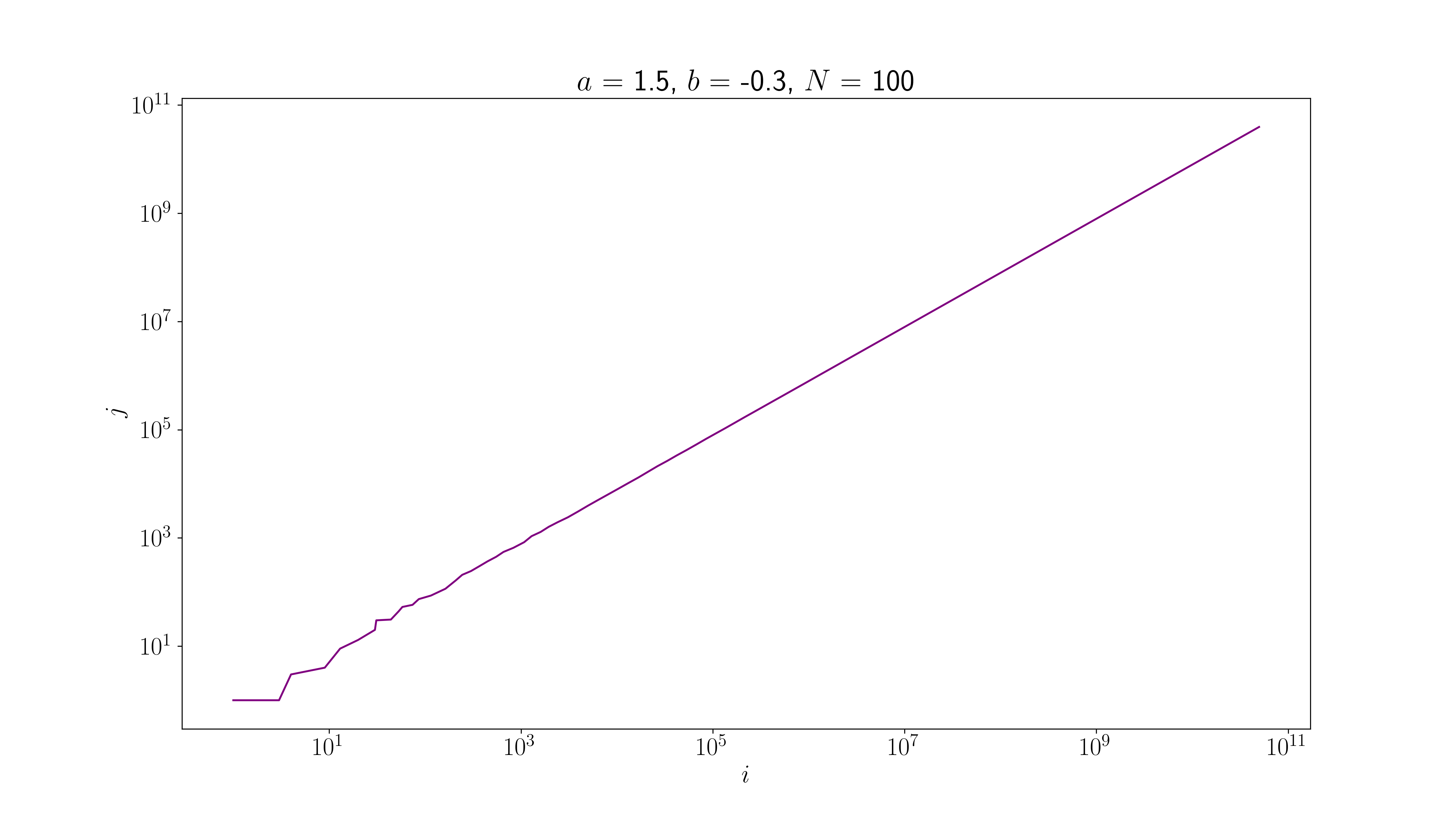}
\caption{\label{transiencebsupone}Log-log plot of a typical trajectory of $(X_n)$, with $a = 1.5, b=-0.3$ and $N=100$}
\end{figure}

For this case, we will take benefit of the comparison between the stochastic process $(\Tilde X_n)$ and its linear deterministic version.
Namely, let us consider the linear recurrence relation defined by $y_0, y_1 \in \N$ and :

\begin{equation}
\label{linear recurrence}
\forall n \geq 0, \quad y_{n+2} = ay_{n+1} + by_n + \lambda.
\end{equation}
The solutions to this equation are determined by the eigenvalues and eigenvectors of the matrix $\begin{pmatrix}
0 & b \\
1 & a
\end{pmatrix}$, which is the companion matrix of the polynomial $X^2 - aX - b$ (see Appendix A. for more details). 
An easy calculation shows that in case T2, we have $a^2+4b > 0$, hence the eigenvalues are simple and real-valued. We denote the largest eigenvalue by:
$$\theta \coloneqq \dfrac{a+\sqrt{a^2+4b}}{2}.$$
The following holds in case T2, as can be easily verified:
\begin{align}
    \label{eq:theta_1} \theta &> 1\\
    \label{eq:theta2+b} \theta^2 + b &> 0.
\end{align}
In fact, one can check that case T2 exactly corresponds to the region in the space of parameters $a,b$ where $\theta>1$, meaning that the sequence $(y_{n+1},y_n)_{n\ge0}$, with $(y_n)_{n\ge0}$ the solution to \eqref{linear recurrence}, grows exponentially inside the positive quadrant, along the direction of the eigenvector $(\theta,1)$.

%
%
In what follows, we fix $1 < r < \theta$, such that 
\begin{equation}
    \label{eq:r}
    r^2-ar-b < 0,
\end{equation}
where we use the fact that $\theta > 1$ is the largest root of the polynomial $X^2 - aX - b$. 


We split our study into two different sub-cases depending on the sign of $b$.

\subsubsection{Subcase T2a : $b \geq 0$}

In this case, we have $a\Tilde X_n + b \Tilde X_{n-1} + \lambda > 0$ for all $n\in\N$, and so $\Tilde X_{n+1} \sim \mathcal{P}(a\Tilde X_n + b\Tilde X_{n-1} + \lambda)$, i.e.,~no truncation is necessary. It is classical that in this case, $\Tilde X_n$ grows exponentially in $n$ almost surely, but we provide a simple proof for completeness. 

We therefore apply Lemma \ref{lem:strategy_transience} with  the sequence $m_n=n$ and 
$$S_n=\{ (i,j)\in\N^2, i\ge r^n, j\ge r^{n-1}\}.$$
With these notations, assumption 1 is automatically satisfied. Assumption 2 is also satisfied, because $(i,j)\to_1 (k,i)$ for every $i,j,k\in\N$, since $ai+bj+\lambda > 0$ for every $i,j\in\N$ as explained above.

In order to prove assumption 3, let us consider $n \in \N$ and let $(i,j)\in S_n$. By definition, starting from $(i,j)$, $\Tilde X_{1} \sim \mathcal{P}(a  i + bj + \lambda)$. Thus,
\begin{align*}
\P_{(i,j)} (\Tilde X_1 < r^{n+1})
&= \proba{\mathcal{P}(ai + bj + \lambda) < r^{n+1}} \\
&\leq  \proba{\mathcal{P}(ar^{n} + b r^{n-1}) < r^{n+1}} \\
&= \proba{\mathcal{P}(ar^{n} + b r^{n-1}) - (ar^{n} + br^{n-1}) < r^{n-1}(r^2-ar - b)}.
\end{align*}

Recall that $r^2-ar-b < 0$ by \eqref{eq:r}, which implies : 
\begin{align*}
\P_{(i,j)} (\Tilde X_1 < r^{n+1}) &\leq \proba{\left| \mathcal{P}(ar^n + b r^{n-1}) - (ar^n + br^{n-1}) \right| >  -r^{n-1}(r^2-ar - b)} \\
&\leq \dfrac{(a+b)r^2}{r^{n}(r^2-ar-b)^2} ,
\end{align*}
where we used again the Bienaymé-Chebychev inequality. Thus $$
\P_{(i,j)} (X_1\in S_{n+1})\ge  \left(1-\dfrac{(a+b)r^2}{r^{n}(r^2-ar-b)^2}\right)_+ \eqqcolon p_n.
$$
This allows to conclude the proof with Lemma \ref{lem:strategy_transience}, as in the previous case.

\subsubsection{Subcase T2b : $b < 0$}

In this case, because of the negativity of $b$ it is more difficult to find an adequate lower-bound of $a\Tilde X_n + b\Tilde X_{n-1}$. We will thus prove a stronger result, which is illustrated on Figure \ref{transiencebsupone} : asymptotically, the process $(\Tilde X_n)$ grows exponentially and the ratio $\Tilde X_{n+1}/\Tilde X_n$ is close to $\theta$.

From \eqref{eq:r} and \eqref{eq:theta2+b}, we can choose $\varepsilon > 0$ small enough such that : 
\begin{align}
\label{eq:T2b_cond_r_eps}
    r^2 - a(r-\varepsilon)-b < 0\\
\label{eq:T2b_cond_theta_eps}
    \theta^2-\theta\varepsilon+b > 0.
\end{align}
We will use Lemma \ref{lem:strategy_transience} using $m_n = n$ and for $n \in \N^*$ : 
$$S_n=\left\{ (i,j)\in\N^2, i\ge r^n, j\ge r^{n-1}, \left|\frac{i}{j}-\theta \right| \leq \varepsilon \right\}.$$
Note that Assumption 1 from Lemma~\ref{lem:strategy_transience} is again automatically verified. Assumption 2 is also verified, since for $(i,j)\in S_n$, we have
\begin{equation}
\label{eq:T2b_param_positive}
a i + b  j + \lambda > \big(a(\theta - \varepsilon)+b\big) j \geq \big(a(r-\varepsilon)+b\big)j > 0,
\end{equation}
by \eqref{eq:T2b_cond_r_eps}, and so $(i,j)\to_1 (k,i)$ for every $k\in \N$.

We now show that Assumption 3 from Lemma~\ref{lem:strategy_transience} is verified. Let $n\in \N$ and $(i,j) \in S_n$. Then :
\begin{align} 
\label{eq:case_T2b}
\P_{(i,j)}(X_{1}\notin S_{n+1} )&\leq \P_{(i,j)}\left(\Tilde X_{1} < r^{n+1} \right) + \P_{(i,j)}\left(\left| \dfrac{\Tilde X_{1}}{i}-\theta\right| > \varepsilon\right).
\end{align}
We first bound the first term on the right-hand side of \eqref{eq:case_T2b}. By \eqref{eq:T2b_param_positive}, we have
\begin{align*}
&\P_{(i,j)}\left( \Tilde X_{1} < r^{n+1}\right) \\
&= \P\left( \mathcal{P}(ai + b j + \lambda) < r^{n+1} \right)\\
&\leq  \P\left(\mathcal{P}\Big(\big[a(r-\varepsilon)+b\big]r^{n-1}\Big) < r^{n+1}\right) \\
&= \P\left(\mathcal{P}\Big(\big[a(r-\varepsilon)+b\big]r^{n-1}\Big)-\big[a(r-\varepsilon)+b\big]r^{n-1} < r^{n-1}\big[r^2 - a(r-\varepsilon)-b\big]\right).\end{align*}

Furthermore, using \eqref{eq:T2b_cond_r_eps} and applying the Bienaymé-Chebychev inequality we obtain :
\begin{equation}
\label{C1ineq}
\begin{aligned}
&\P_{(i,j)}\left(\Tilde X_{1} < r^{n+1}\right) \\
&\leq \P\left(\left| \mathcal{P}\Big(\big[a(r-\varepsilon)+b\big]r^{n-1}\Big)-\big[a(r-\varepsilon) + b\big]r^{n-1} \right| > -r^{n-1}\big[r^2 - a(r-\varepsilon)-b\big]\right) \\
&\leq \dfrac{\big[a(r-\varepsilon)+b\big]r^{n-1}}{\Big[r^{n-1}\big[r^2 - a(r-\varepsilon)- b\big]\Big]^2} \\
&= \dfrac{\big[a(r-\varepsilon)+b\big]r}{r^n\big[r^2 - a(r-\varepsilon)- b\big]^2}= \dfrac{C_1}{r^n},
\end{aligned}
\end{equation}
where $C_1$ is a constant that does not depend on $n$.
\vspace{0.3cm}

We now bound the second term on the right-hand side of \eqref{eq:case_T2b}. 
Let us write :
\begin{align*}
\left| \dfrac{\Tilde X_{1}}{i}-\theta\right| &= \left| \dfrac{\Tilde X_{1}-\E_{(i,j)}[\Tilde X_{1} ]}{i}\right| + \left| \dfrac{\E_{(i,j)} [\Tilde X_{1} ]}{i}-\theta \right|
\end{align*}
First notice that, for any $(i,j) \in S_n$ :
\begin{align}
\left| \dfrac{\E_{(i,j)}\left[\Tilde X_{1} \right]}{i}-\theta \right| &= \left| \dfrac{ai +bj+\lambda}{i}-\theta \right| 
= \left| a +b \dfrac{j}{i}+ \dfrac{\lambda}{i}-\theta \right| \nonumber\\
&\leq \underbrace{\left|a+\frac{b}{\theta}-\theta\right|}_{= 0} + |b| \left| \dfrac{j}{i} - \frac{1}{\theta} \right| + \frac{\lambda}{i} \nonumber\\
&< \dfrac{|b|}{\theta(\theta-\varepsilon)} \varepsilon + \frac{\lambda}{i},\label{eq:maj_moyenne}
\end{align}
where we used that if $|x-\theta|<\varepsilon$ and $\varepsilon < \theta$, then
\[
\left|\frac 1 x - \frac 1 \theta\right| = \frac{|\theta-x|}{x\theta} < \frac{\varepsilon}{\theta(\theta-\varepsilon)}.
\]
To prove that $\P_{(i,j)}\left(\left| \dfrac{\Tilde X_{1}}{i}-\theta\right| > \varepsilon\right) \leq \dfrac{C_2}{r^n}$, where $C_2$ is a constant that does not depend on $n$, we deduce from \eqref{eq:maj_moyenne} that it is sufficient to show that  :
$$\P_{(i,j)}\left( \dfrac{\Big| \Tilde X_{1}-\E_{(i,j)}[\Tilde X_{1}] \Big| + \lambda}{i} > \delta \varepsilon \right) \leq \frac{C_2}{r^n},$$
where, by \eqref{eq:T2b_cond_theta_eps}, we have
\[
\delta \coloneqq 1- \dfrac{|b|}{\theta(\theta-\varepsilon)} > 0.
\]
Furthermore, since $b<0$ and $(i,j)\in S_n$, we have that 
$$ai + bj +\lambda \leq  ai + \lambda \leq (a+\lambda)i.$$
We finally have, using the Bienaymé-Chebyshev inequality : 
\begin{align*}
\P_{(i,j)}\left(\dfrac{\Big| \Tilde X_{1}-\E_{(i,j)}[\Tilde X_{1}] \Big| + \lambda}{i} > \delta \varepsilon \right) 
&=\P_{(i,j)} \left( \Big| \Tilde X_{1}-\E_{(i,j)}[\Tilde X_{1}] \Big| > \delta \varepsilon i-\lambda \right) \\
&\leq\dfrac{(a+\lambda)i}{(\delta \varepsilon i-\lambda)^2}\\
&= \dfrac{a+\lambda}{\left( \delta \varepsilon \sqrt{i} - \lambda/ \sqrt{i}\right)^2} \color{black}\\
&\leq \dfrac{a+\lambda}{\left(\delta \varepsilon r^{n/2}-\lambda r^{-n/2}\right)^2} 
\end{align*}

Last inequality holds for a sufficiently large $n$. Indeed, since $i \geq r^n$, we always have 
$$\delta \varepsilon \sqrt{i} - \lambda/ \sqrt{i} > \delta \varepsilon r^{n/2}-\lambda r^{-n/2},$$ 
and for $n$ large enough we have 
$$\delta \varepsilon \sqrt{i}- \frac{\lambda}{\sqrt{i}} > 0.$$
\color{black}
This yields, for some constant $C_2<\infty$,
\begin{equation}
\label{C2ineq}
\P_{(i,j)} \left( \dfrac{\Big| \Tilde X_{1}-\E_{(i,j)}[\Tilde X_{1}] \Big| + \lambda}{i} > \left( 1- \dfrac{|b|}{\theta(\theta-\varepsilon)} \right) \varepsilon \right) \leq \frac{C_2}{r^n}.
\end{equation}
Combining \eqref{C1ineq} and \eqref{C2ineq} we have that :
$$\P_{(i,j)}(X_{1}\in S_{n+1}) \ge \left( 1- \dfrac{C_1+C_2}{r^n} \right)_+ =: p_n.$$
Which will finally lead us to the result, by using Lemma \ref{lem:strategy_transience} as before.

\section{Perspectives and open problems}

\subsubsection*{Critical behavior}
In the case of linear Hawkes processes, it is well known that, at criticality, the process achieves fractal-like, i.e.,~heavy-tail behavior, related to critical branching processes. 
It is tempting to believe that this should remain true on the whole boundary between the phases $\mathcal R$ and $\mathcal T$, but the fractal exponents might differ.

For the sake of completeness, we offer a numerical study of the various critical cases of the considered model, which indicates different behaviour depending whether $a<2$ or $a>2$. We present realizations of the process $(\Tilde X_n)$, as we believe it is simpler to visualize the behavioral differences compared to showing realizations of the Markov chain in $\mathbb N^2$. Given the diversity of the process behaviours, we anticipate the need for various probabilistic tools to describe the process evolution over long time spans. We consider the same setting as for the previous figures: an initial condition $\Tilde X_{-1}, \Tilde X_0 = 0$ and $\lambda = 1$. The number $N$ describes the number of simulated steps.
On Figure \ref{fig:critical_R1}, we observe a linear growth of the discrete time process $\widetilde{X}_n$, with oscillations to $0$ when $a<0$ and $b=1$ (left panel) and without oscillations in the case $a+b=1$.
\begin{figure}[h!]
    \centering
    \includegraphics[width=0.49\linewidth]{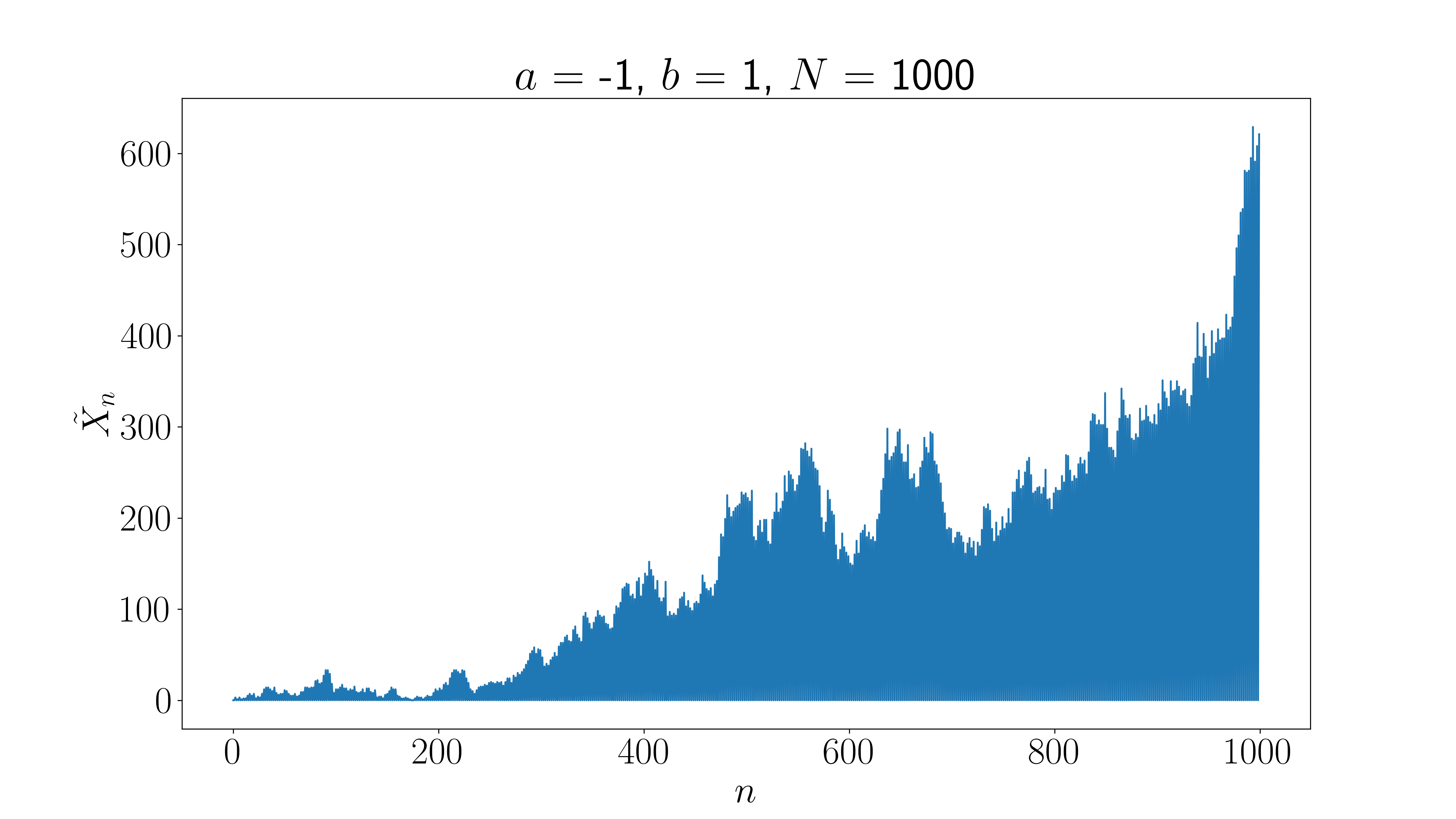}
    \includegraphics[width=0.49\linewidth]{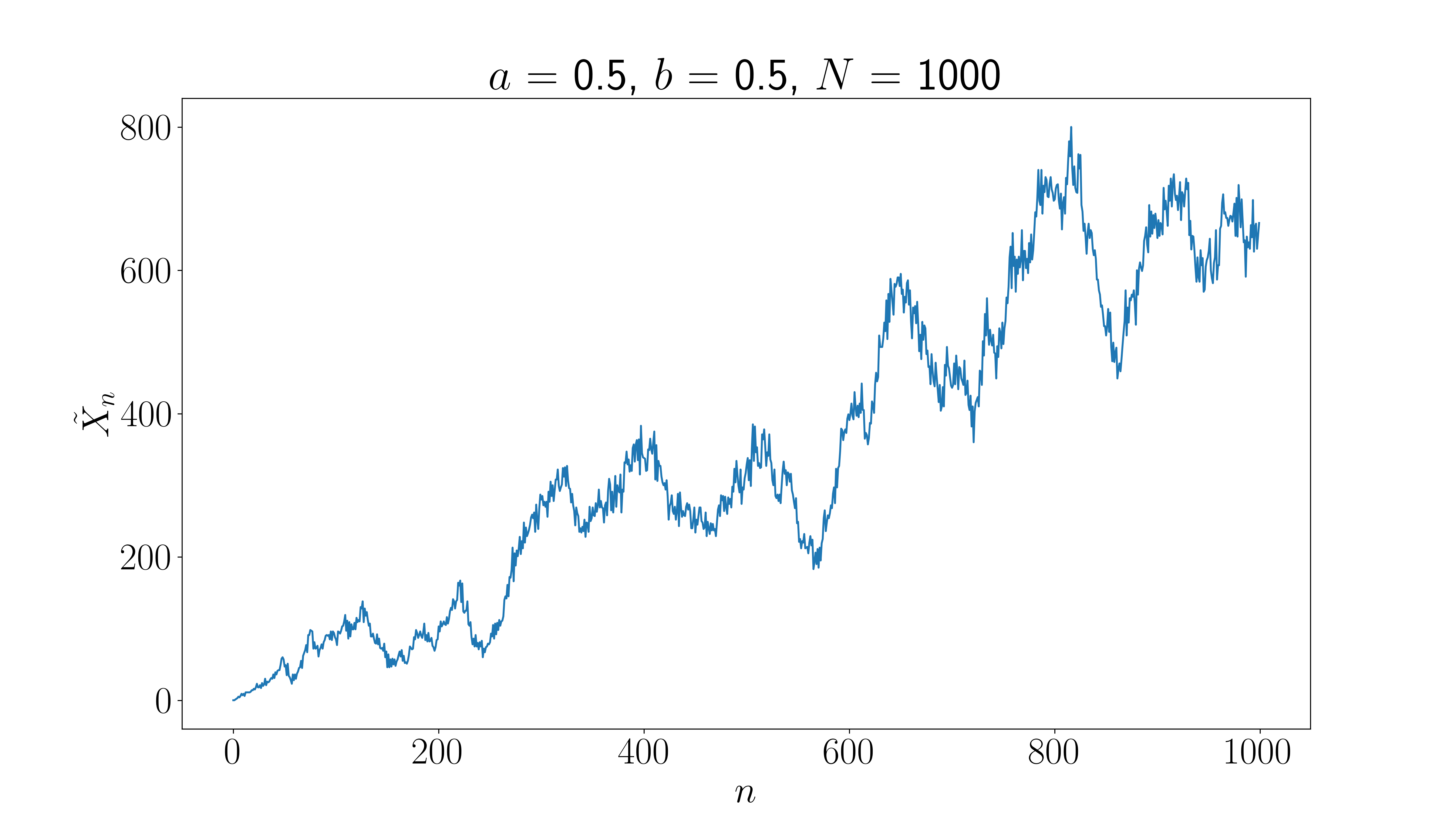}
    \caption{Trajectories of $(\widetilde{X}_n)_{0\le n\le 1000}$ for critical parameters $(a,b)=(-1,1)$ on the left and $(a,b)=(0.5,0.5)$ on the right. We observe a linear growth of the process as it could be expected in critical cases.}
    \label{fig:critical_R1}
\end{figure}
The situation seems to be different for $a\ge2$ and $b=-a^2/4$. When $a>2$ we observe an exponential growth (Fig. \ref{fig:critical_R2} (left)) similarly to the transient regime, while the case $a=2$ presents large excursions away from $0$ but deciphering transient or recurrent behavior is difficult. These simulations show that the study of these critical cases is an interesting research topic for the future.
\begin{figure}[h!]
    \centering
    \includegraphics[width=0.49\linewidth]{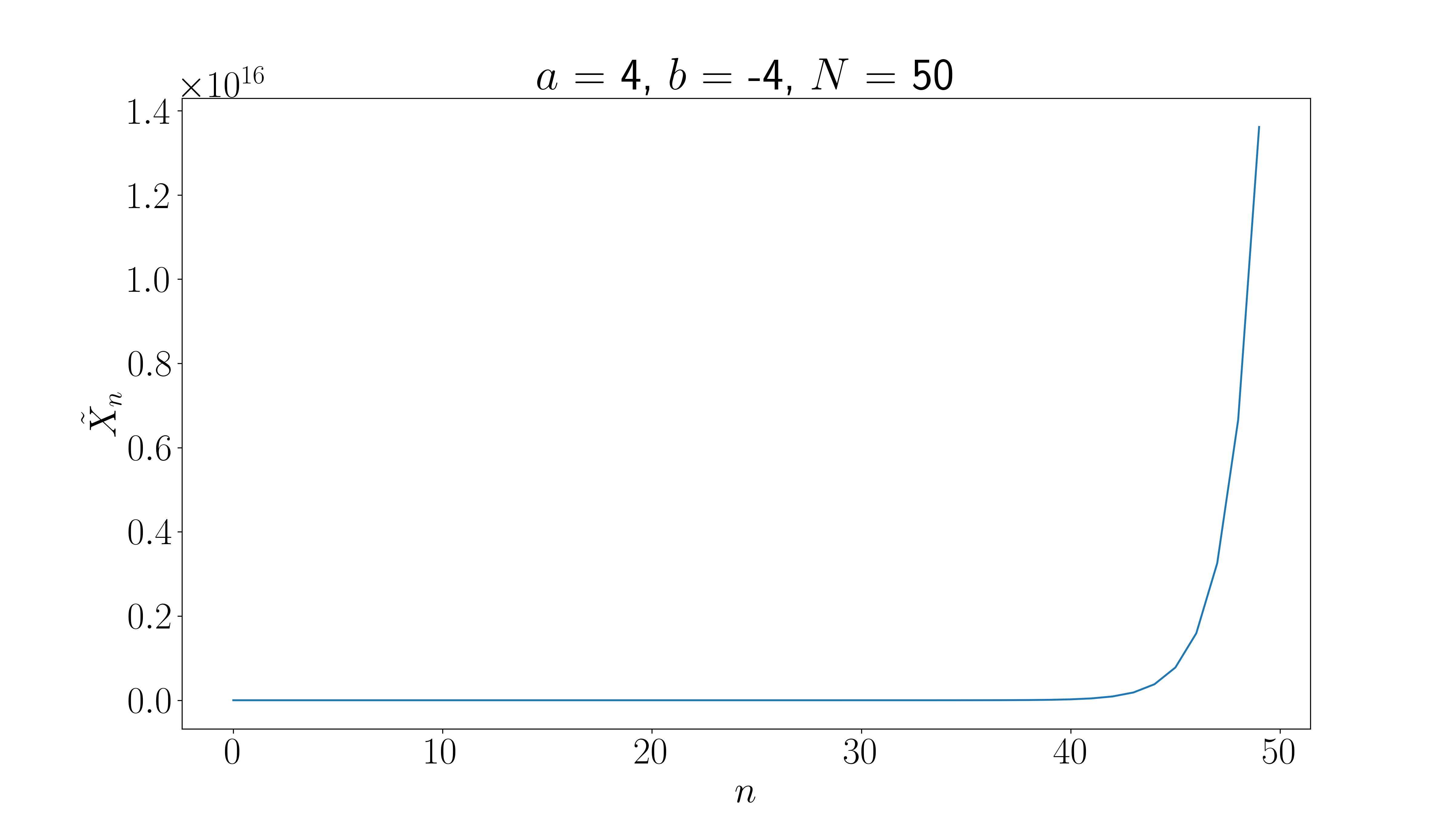}
    \includegraphics[width=0.49\linewidth]{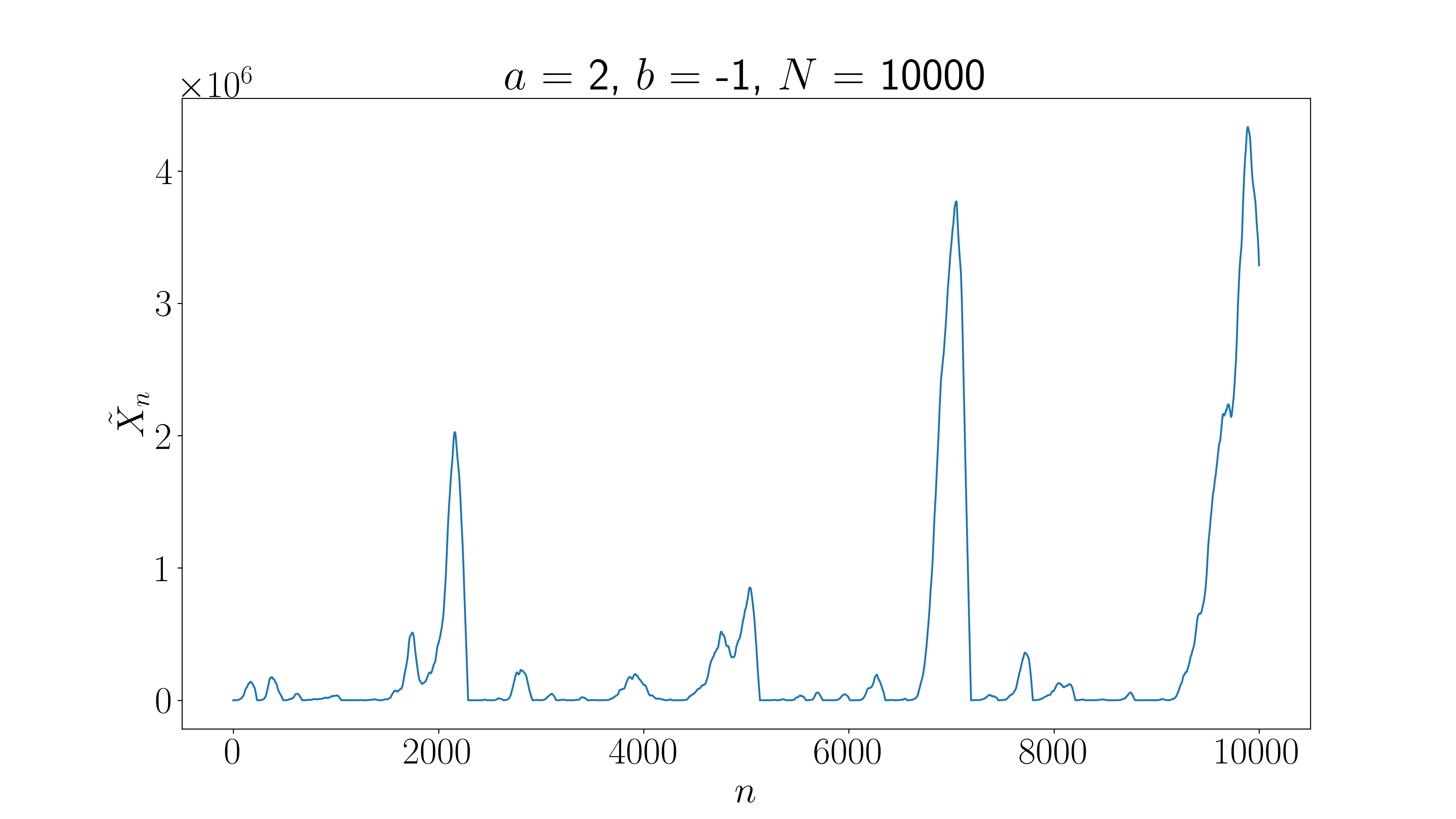}
    \caption{Trajectories of $(\widetilde{X}_n)_{0\le n\le 1000}$ for critical parameters $(a,b)=(4,-4)$ on the left and $(a,b)=(2,-1)$ on the right.}
    \label{fig:critical_R2}
\end{figure}

\subsubsection*{Generalization of the model} As explained at the beginning of the article, the results obtained here should be seen as a starting point for the search for necessary and sufficient conditions for the stability of Hawkes processes with inhibition, in discrete or continuous time. 


We believe that obtaining a similar classification in the cases $p>2$ or $p=\infty$ is a very difficult problem. 
It should be closely related to the study of the asymptotic behavior of certain deterministic equations, such as the non-linear recurrence equation
\[
x_n = (a_1x_{n-1}+\cdots+a_px_{n-p})_+.
\]
It seems that the algebraic structures underlying these equations are intricate and, to this date, unknown. The understanding of these structures seem crucial for the study of the asymptotic behavior of the solutions to these equations.

\appendix

\section{Linear recurrence equations}

Let $\alpha \in \R$, $p\in \N$ and $a_1,\ldots,a_p \in \R$.
Consider the linear recurrence equation
\[
x_n = a_1x_{n-1} + \cdots + a_px_{n-p} + \alpha,\quad n\ge 1,
\]
with given initial data $x_{0},\ldots,x_{-p+1}\in \R$. Define the matrix 
\[
A = \begin{pmatrix}
a_1 &  \cdots & a_{p-1} & a_p\\
1 &  & &\\
 & \ddots & & \\
& & 1 & 
\end{pmatrix},
\]
where vanishing entries are meant to be zero.
Then, setting
\[
\bar x_n = \begin{pmatrix}
x_n\\
\vdots\\
x_{n-p+1}
\end{pmatrix}\quad \text{and} \quad \bar \alpha = \begin{pmatrix}\alpha\\0\\ \vdots \\ 0\end{pmatrix},
\]
the sequence $(\bar x_n)_{n\ge 1}$ solves the system of linear recurrences
\[
\bar x_n = A \bar x_{n-1} + \bar \alpha ,\quad n\ge 1.
\]
Recall that the spectral radius $\rho(A)$ of the matrix $A$ is the defined by
\[
\rho(A) = \max(|\theta_1|,\ldots,|\theta_p|),
\]
where $\theta_1,\ldots,\theta_p\in \C$ are the complex eigenvalues of $A$, counted with algebraic multiplicity. Equivalently, $\theta_1,\ldots,\theta_p$ are the roots, counted with multiplicity, of the characteristic polynomial
\[
P(z) = \det(zI - A) = z^p - a_1z^{p-1} - \cdots -a_p. 
\]

We recall the following classical fact:
\begin{theorem}[\cite{gallier2020linear} Chapter 9, Thm 9.1]
\label{th:lin_rec}
The following are equivalent:
\begin{enumerate}
    \item $\bar x_n$ converges as $n\to\infty$, for every initial data $x_0,\ldots,x_{-p+1}$
    \item $\rho(A) < 1$
\end{enumerate}
\end{theorem}

In the case $p=2$, setting $a = a_1$ and $b=a_2$, we have $\P(z) = z^2 - az-b$. Its roots are
\[
\theta_\pm = \frac a 2 \pm \sqrt{\frac {a^2} 4 + b}.
\]
In particular,
\[
\rho(A) = \begin{cases}
\frac {|a|} 2 + \sqrt{\frac {a^2} 4 + b}, & \text{if }\frac {a^2} 4 + b \ge 0\\
\sqrt{-b}, &  \text{if }\frac {a^2} 4 + b  < 0.
\end{cases}
\]
A quick calculation shows that
\begin{equation}
\label{eq:rho_a_p=2}
\rho(A) < 1 \iff |a| + b < 1\text{ and } b > -1.
\end{equation}

This corresponds to the triangular dashed region of parameters in Figure \ref{fig:RT} of section 2. 
\color{black}

\section{Criteria for strong irreducibility}
\label{sec:proof_irr}

The Markov chain considered in this article is irreducible in the (weak) sense of Douc, Moulines, Priouret and Soulier \cite{dmps_markov_chains}, but not necessarily strongly irreducible, i.e.,~irreducible in the classical sense. In this section, we study the decomposition of the state space into communicating classes. We recall the basic definitions. Let $x,y\in \N^2$. We say that $x$ \emph{leads to} $y$, or, in symbols, $x \to y$, if there exists $n\ge 0$ such that $\P(X_n = y\mid X_0 =x) > 0$. We say that $x$ \emph{communicates with} $y$ if $x\to y$ and $y\to x$. This is an equivalence relation which partitions the state space $\N^2$ into classes called \emph{communicating classes}. 
 
Recall that the Markov chain is called \emph{strongly irreducible} if all states are accessible, equivalently, if $\N^2$ is a communicating class. A communicating class $C\subset \N^2$ is called \emph{closed} if there does not exist $x\in C$ and $y\in C^c$, such that $x\to y$. 



\begin{prop}\ 
\label{prop:irr}
\begin{itemize}
    \item The Markov chain $(X_n)$ is strongly irreducible on $\mathbb{N}^2$ if and only if $a \geq 0$, or  if $a > -\lambda$ and $a+b \geq 0$.
    \item The communicating class of $(0,0)$ contains
\begin{equation}
\label{setS}
\mathcal{S} = \left\{ \couple{0}{0} \right\} \cup \left\{ \couple{0}{k}, k\in \mathbb{N}^* \right\} \cup \left\{ \couple{k}{0}, k\in \mathbb{N}^* \right\}.
\end{equation}
and is actually equal to $\mathcal{S}$ iff $a\le -\lambda$.
\end{itemize}
\end{prop}

We will use the following :
\begin{lemma}
\label{P^n}
Let $i,j,k,\ell \in \mathbb{N}$. The transition matrix $P$ of the Markov chain $(X_k)$ satisfies : 
$$P^2\left(\couple{i}{j}, \couple{k}{\ell}\right) = \dfrac{e^{-(s_{ij}+s_{\ell i})}s_{ij}^\ell s_{\ell i}^k}{\ell! ~ k!},$$
and for all $n \geq 3$, 
\begin{equation}
\label{P^nformula}
P^n\left(\couple{i}{j}, \couple{k}{\ell} \right) =  \mathlarger{\mathlarger{\sum}}_{m_1, \dots, m_{n-2} \in \mathbb{N}} \dfrac{\exp\left\{-\displaystyle \sum_{q=1}^n s_{\sigma_{q+1}^n \sigma_{q+2}^n}\right\} \displaystyle \prod_{q=1}^n s_{\sigma_{q+1}^n \sigma_{q+2}^n}^{\sigma_{q}^n }}{m_1! \dots m_{n-2}! ~k! ~\ell! },
\end{equation}
with $\sigma^n := (\sigma^n_1, \sigma^n_2, \dots, \sigma^n_{n+2}) = (k, \ell, m_{n-2}, \dots, m_1, i, j)$.
\end{lemma}

\begin{proof}[Proof of proposition \ref{prop:irr}]
As mentioned above:
$$\couple{i}{j} \to \couple{0}{i} \to \couple{0}{0},$$
for any $(i,j) \in \N^2$ since it only requires that 2 successive 0 are drawn from the Poisson random variable. Therefore, to prove strong irreducibility, it is thus sufficient to prove that $\couple{0}{0} \leadsto \couple{i}{j}$, for all $\couple{i}{j} \in \N^2$. Let us consider different cases, depending on the values of the parameters $a$ and $b$.

\paragraph{\underline{ $-$ If $a \geq 0$ :} }

Since $\lambda >0, s_{00}>0$ thus $\couple{j}{0}$ is accessible from $\couple{0}{0}$, for all $j\in \N$.

Moreover, when $a \geq 0$, $s_{j0} = (aj + \lambda)_+ > 0$ and then $\couple{j}{0} \to \couple{i}{j}$, yielding the result.\medskip

\paragraph{\underline{$-$ If $-\lambda < a < 0$ and $a+b \geq 0$ :}}

Let $k \in \mathbb{N}$. Since $a+b \geq 0$ and $a+\lambda >0$, we have :
$$s_{k+1,k} = (a(k+1) + bk + \lambda)_+ = ((a+b)k + a + \lambda)_+ > 0,$$

Let $\couple{i}{j} \in \N^2$. Since $s_{k+1,k}>0$ for all $k$, we deduce that any $(\ell, k+1)$ is accessible from $(k+1,k)$. Thus, in order to reach $(i,j)$ from $(0,0)$, we move from small steps to $(j,j-1)$, and then reach $(i,j)$ : 
$$\couple{0}{0} \to \couple{1}{0} \to \couple{2}{1} \to \dots \to \couple{j}{j-1} \to \couple{i}{j}.$$
which concludes the proof of this case.\medskip

\paragraph{ \underline{$-$ If $a \leq -\lambda$:}}

We will prove that the communicating class of $(0,0)$ is given by \eqref{setS}. \\Let $k \in \mathbb{N}^*$, then as previously, we have that $\couple{0}{0} \to \couple{k}{0}$ since $s_{00}>0$, however since $a\leq -\lambda$, $s_{k0} = (ak+\lambda)_+ = 0$, and the next step of the Markov chain will be $\couple{0}{k}$. 
Depending on the value of the parameter $b$, the next step of the Markov chain will either be $\couple{0}{0}$ if $s_{0k}=0$, or $\couple{k'}{0}$ with $k'\geq 0$ if $s_{0k}>0$ and so on. 
This proves that the class $cl(0,0)$ is closed and given by \eqref{setS}.

\paragraph{\underline{ $-$ If $-\lambda < a < 0$ and $a+b < 0$}}
In this case we can only prove that the Markov chain is not strongly irreducible on $\N^2$ but we do not identify the communicating class of $(0,0)$.
\begin{itemize}
\item\textbf{Case 1 :  $b\le 0$.}

Since $a<0$, we can choose $k_\star$ such that
$$ak_\star + \lambda \leq 0.$$
We will show that it is not possible to reach the state $\couple{1}{k_\star}$.
Assuming the opposite leads to the existence of $\ell \in \mathbb{N}$ such that 
$\couple{k_\star}{\ell} \to \couple{1}{k_\star}$
which implies that $s_{k_\star, \ell} > 0$. If $b<0$, we deduce that necessarly
$$ak_\star + b\ell + \lambda > 0 \Longrightarrow \ell < \dfrac{-ak_\star - \lambda}{b} \leq 0,$$
so $\ell < 0$ which is contradictory. 
We then deduce that the Markov chain is reducible. 

If $b=0$, $s_{k_\star, \ell} > 0$ would imply that $ak_\star + \lambda > 0$ which contradicts the definition of $k_\star$.
\item\textbf{Case 2 :  $b>0$.}

Since $a+b < 0$, it is possible to choose $k_\star \in \mathbb{N}$ large enough so that : 
$$(a+b)k_\star + \lambda \leq 0.$$ 
In particular, $0 \geq ak_\star+bk_\star+\lambda \geq ak_\star+\lambda$, so that :
$$\dfrac{-ak_\star - \lambda}{b} \geq k_\star \geq \dfrac{-\lambda}{a}.$$
Notice that $k_\star \geq 2$ since $k_\star \geq \dfrac{-\lambda}{a} > 1$.

We will show that it is not possible to reach $\couple{1}{k_\star}$ starting from $\couple{0}{0}$. 
Assuming the opposite leads us to the existence of $n \in \mathbb{N}$ such that 
$$P^n\left( \couple{0}{0}, \couple{1}{k_\star} \right) > 0.$$ 
Using \eqref{P^nformula} in the Lemma \ref{P^n} implies that it exists $m_1, \dots, m_{n-2} \in \mathbb{N}$ such that :  
$$\left\{ \begin{array}{ll} s_{k_\star,m_{n-2}} > 0 \vspace{0.1cm}\\ s_{m_{n-2} m_{n-3}}^{k_\star} > 0 \vspace{0.1cm}\\ \dots \vspace{0.1cm}\\ s_{m_2 m_1}^{m_3} > 0 \vspace{0.1cm}\\ s_{m_1 0}^{m_2} > 0 \end{array} \right. $$
First, we thus have that : 
$$ak_\star + bm_{n-2} + \lambda > 0 \Longrightarrow m_{n-2} > \dfrac{-ak_\star - \lambda}{b} \geq k_\star,$$ 
then, since $k_\star > 0$, we necessarily have $s_{m_{n-2}m_{n-3}}>0$. It yields :
$$am_{n-2} + bm_{n-3} + \lambda > 0 \Longrightarrow m_{n-3} > \dfrac{-a m_{n-2} - \lambda}{b} \geq \dfrac{-ak_\star - \lambda}{b} \geq k_\star.$$ 
We thus have by immediate induction that :
$$\forall i \in \{ 1, \dots,n-2 \}, ~m_i \geq k_\star \geq \dfrac{-\lambda}{a}.$$
Finally, $s_{m_1,0} > 0$ implies $a m_1 + \lambda > 0$, which is contradictory.

We conclude that there is no finite path between $\couple{0}{0}$ and $\couple{1}{k_\star}$, so the Markov chain $(X_k)_{k\ge0}$ is reducible on $\mathbb{N}^2$.
\end{itemize}
\end{proof}

\section*{Acknowledgements}
We thank two anonymous reviewers for their valuable suggestions, which helped to improve the presentation of the paper.

M.C. has been supported by the Chair ”Modélisation Mathématique et Biodiversité” of Veolia Environnement-École Polytechnique-Muséum national d’Histoire naturelle-Fondation X and by ANR project HAPPY (ANR-23-CE40-0007) and DEEV (ANR-20-CE40-0011-01). P.M. acknowledges partial support from ANR grant ANR-20-CE92-0010-01 and from Institut Universitaire de France.

\bibliographystyle{plain}
\bibliography{bibliography}

\end{document}